\documentclass[11pt,hyp,]{nyjm}
\usepackage{hyperref}
\hypersetup{nesting=true,debug=true,naturalnames=true}
\usepackage{graphicx,amssymb,upref}
\usepackage{xfrac}
\usepackage[vcentermath]{youngtab}
\usepackage{mathdots} 
\usepackage[all,cmtip]{xy} 
\usepackage{float} 
\usepackage{graphicx,enumerate,color}
\usepackage{tikz} 
\usepackage{tikz-cd}
\usepackage{hyperref}
\hypersetup{colorlinks,
citecolor=blue,
filecolor=green,
linkcolor=blue,
urlcolor=blue
}

\newtheorem{thm}{Theorem}
\newtheorem{cor}[thm]{Corollary}
\newtheorem{defn}[thm]{Definition}
\newtheorem{lem}[thm]{Lemma}

\let\<\langle
\let\>\rangle

\let\uml\"

\newcommand{\R}{\mathbb R}
\newcommand{\C}{\mathbb C}
\newcommand{\Z}{\mathbb Z}
\newcommand{\F}{\mathbb F}
\newcommand{\triv}{_{\operatorname{triv}}}

\newcommand{\Gr}{\operatorname{Gr}}
\renewcommand{\S}[2]{S^{#1,#2}}
\newcommand{\Spq}{\S pq}
\newcommand{\D}[2]{D^{#1,#2}}
\newcommand{\pt}{\operatorname{pt}}

\newcommand{\rH}[2]{\tilde H^{#1,#2}}
\renewcommand{\H}[2]{H^{#1,#2}}

\newcommand{\Mt}{\mathbb{M}_2}
\newcommand{\M}[2]{\Sigma^{#1,#2}\Mt	}
\newcommand{\Zt}{\sfrac{\Z}{2}}

\newcommand{\Ct}{C_2}

\renewcommand{\part}{\operatorname{part}}
\definecolor{ggreen}{rgb}{0,.5,0}

\newcommand{\surj}{\twoheadrightarrow}
\newcommand{\inj}{\hookrightarrow}
\newcommand{\RP}{\R\text{P}}  

\newcommand{\sgn}{_{\operatorname{sgn}}}
\newcommand{\rs}{\operatorname{rowspace}}
\newcommand\srac[2]{\genfrac{}{}{0pt}{}{#1}{#2}}
\renewcommand{\H}[2]{H^{#1,#2}}
\newcommand{\Hbb}{\H\bullet\bullet}
\newcommand{\rank}{\operatorname{rank}}
\newcommand{\Smash}             {\wedge}

\newcommand{\iso}               {\cong}
\DeclareMathOperator{\Top}{Top}

\DeclareMathOperator{\trace}{trace}
  
\newcommand{\ds}{\displaystyle}

\DeclareMathOperator{\cok}{cok}
\newcommand{\jol}{Jack-o-lantern }
\newcommand{\HMtwo}[2]{
	  		\def\x{#1}
	  		\def\y{#2}
	  		\draw [->] (\x,\y) -- (\x,4.5);
	  		\draw [->] (\x,\y) -- (\x+4.5-\y,4.5);
	  		\draw [->] (\x,\y-2) -- (\x,-4.5);
	  	    \draw [->] (\x,\y-2) -- (\x-\y-2.5,-4.5);
}
\newcommand{\axisname}[1]{
			\draw (1,-5) node{#1};
	  	    \draw [dotted] (-1,0) -- (4,0);
	  	    \draw [dotted] (0,-4) -- (0,4);
	  	    \foreach \x in {1,...,3}
				\draw(\x,0)node[below]{$\x$};
	  	    \foreach \y in {1,...,4}{
				\draw(0,\y)node[left]{$\y$};
				\draw(0,-1*\y)node[left]{$-\y$};
			}
}
\newcommand{\justaxis}{
	  	    \draw [dotted] (-1,0) -- (4,0);
	  	    \draw [dotted] (0,-4) -- (0,4);
}

\title[$RO(\Ct)$-graded cohomology of Grassmannians]{$RO(\Ct)$-graded cohomology of equivariant Grassmannian Manifolds}

\author{Eric Hogle}  
\address{Department of Mathematics, Gonzaga University, Spokane, WA 99202} 
\email{hogle@gonzaga.edu}  


\subjclass[2010]{}

\begin{document} 
 
\begin{abstract}
We compute the $RO(\Ct)$-graded Bredon cohomology of certain families of real and complex $\Ct$-equivariant Grassmannians.
\end{abstract}

\maketitle
\tableofcontents


\section{Introduction}

If $V$ is a representation of the cyclic group $\Ct$, then the Grassmannian $\Gr_k(V)$ inherits a $\Ct$ action. We wish to compute the $RO(\Ct)$-graded Bredon cohomology of these equivariant spaces for various $k$ and $V$. In this paper we present formulas for the cohomologies of two infinite families of finite Grassmannians on real representations, their complex analogs, and also the cohomologies of analogous infinite-dimensional spaces. To do this we use an equivariant version of the Schubert cell construction, giving an equivariant cellular spectral sequence. In general, the differentials in such a spectral sequence are unknown. However, we find convenient situations whose differentials are actually manageable.

We will focus mostly on the real case, postponing complex Grassmannians until Section \ref{complexsection}. The group $\Ct$ has two irreducible real representations: $\R_{\triv}$ with trivial action, and $\R\sgn$ on which the nontrivial group element acts as multiplication by $-1$. The $RO(C_2)$-graded cohomology can therefore be regarded as bigraded. Let  $$\R^{p,q}=(\R_{\triv})^{p-q}\oplus(\R\sgn)^q.$$ Our cohomology theory is graded by both actual and virtual representations, so that a space $X$ with a $\Ct$-action has cohomology groups $\H pq(X;\underline{M})$ for any integer values of $p$ and $q$ and any Mackey functor $\underline{M}$. We will refer to $p$ and $q$ as the \textbf{topological dimension} and \textbf{weight}, respectively, and will sometimes use $|x|$ and $w(x)$ to denote the topological dimension and weight of a pair $x=(p,q)$. We will also refer to the \textbf{fixed-set dimension}, $p-q=|x|-w(x)$.\par 

We denote the one-point compactification of a representation by $S^{p,q}=\widehat\R^{p,q}$, whose underlying space is a $p$-sphere and whose fixed set is a ($p-q$)-sphere (hence the definition of fixed-set dimension above). 
We will be using the constant $\Zt$-valued Mackey functor throughout (analogous to $\Zt$ coefficients in singular cohomology), but these coefficients will be suppressed in the notation; we will write $\H pq(X)$ rather than $H^{\R^{p,q}} (X;\,\underline{\Zt})$. Note that $RO(\Ct)$-graded Bredon cohomology has a bigraded suspension isomorphism with respect to these representation spheres:
$$\rH\bullet\bullet(\S pq \Smash X)\iso \rH{\bullet-p}{\bullet-q}(X).$$ 
Non-equivariant singular cohomology will also appear, and similarly $H^\ast_{\text{sing}}(X)$ will always mean $H^\ast_{\text{sing}}(X;\,\Zt)$. \par

Let $\Gr_k(\R^{p,q})$ denote the manifold of $k$-planes in $p$-dimensional real space, with $C_2$-action induced by that on $\R^{p,q}$. We are interested in calculating $\H\bullet\bullet(\Gr_k(\R^{p,q}))$ as a module over $\Mt:=\H\bullet\bullet(\text{pt})$, the cohomology of a point. Because these spaces can be constructed from representation discs (as we will show in Section \ref{we2} using Schubert cells) their cohomology is known to be a free $\Mt$-module (see \cite{kronholm} or \cite{buddies}) comprised of suspensions $\Sigma^{a,b}\Mt=\rH\bullet\bullet(\S ab)$. And so

$$\H\bullet\bullet(\Gr_k(\R^{p,q}))=\bigoplus_{i}\Sigma ^{a_i,b_i}\Mt$$
where the total number of summands in topological degree $d$ is the rank of non-equivariant singular cohomology for the underlying space:
$$\#\{i:a_i=d\}=\rank H_{\text{sing}}^d(\Gr_k(\R^{p})).$$

However the associated weights $b_i$ were previously known in only a few easy cases. We produce formulas for more families of Grassmannians, namely those of the form $\Gr_k(\R^{n,1})$ and $\Gr_2(\R^{n,2})$. It should be noted that while in the non-equivariant case the Schubert-cell construction gives a chain complex with zero differentials, things will not be so simple here. Whether we progressively compute cohomologies of subspaces using cofiber sequences, or run a single spectral sequence for the Schubert cell filtration, we will in general see many nonzero differentials. \par

\subsection{Preliminaries}

The ground ring $\Mt$ of our theory is non-Noetherian, comprised of a polynomial subalgebra $\Zt[\rho,\tau]$ generated by elements $\rho\in\H11(\text{pt})$ and $\tau\in\H01(\text{pt})$, an element $\theta\in\H0{-2}(\text{pt})$ such that $\theta\rho=\theta\tau=\theta^2=0$, and also an infinite family of elements denoted $\frac{\theta}{\rho^i\tau^j}$ with the property that when $i'\le i$ and $j'\le j$, as the notation suggests, $\rho^{i'}\tau^{j'}\cdot \frac{\theta}{\rho^i\tau^j}=\frac{\theta}{\rho^{i-i'}\tau^{j-j'}}$. 
We will want to draw pictures of this ring.

\begin{figure}[h]
	\begin{tikzpicture}[scale=0.5]
		\justaxis
		\draw(0,0) node {$1$};
		\draw(0,1) node {$\tau$};
		\draw(1,1) node {$\rho$};
		\draw(0,2) node {$\tau^2$};
		\draw(1,2) node {$\rho\tau$};
		\draw(2,2) node {$\rho^2$};
		\draw(0,3) node {$\tau^3$};
		\draw(3,3) node {$\rho^3$};
		\draw(4,4) node {$\iddots$};
		\draw(0.5,4) node {$\vdots$};
		\draw(1.5,4) node {$\vdots$};
		\draw(2.5,4) node {$\iddots$};
		\draw(0,-2) node {$\theta$};
		\draw(0,-3) node {$\frac\theta\tau$};
		\draw(-1,-3) node {$\frac\theta\rho$};
		\draw(-2,-5) node {$\iddots$};
		\draw(-1,-5) node {$\vdots$};
		\draw(-2,-4) node {$\frac\theta{\rho^2}$};
		\draw(0,-4) node {$\frac\theta{\tau^2}$};
		\draw(-3,-5) node {$\iddots$};
		\draw(-0.25,-5) node {$\vdots$};
	\end{tikzpicture}
	\begin{tikzpicture}[scale=0.5]
		\axisname{}
		\HMtwo{0}{0}
		\draw(0,0) node {$\bullet$};
		\draw(0,1) node {$\bullet$};
		\draw(1,1) node {$\bullet$};
		\draw(0,2) node {$\bullet$};
		\draw(1,2) node {$\bullet$};
		\draw(2,2) node {$\bullet$};
		\draw(0,3) node {$\bullet$};
		\draw(1,3) node {$\bullet$};
		\draw(2,3) node {$\bullet$};
		\draw(3,3) node {$\bullet$};
		\draw(0,4) node {$\bullet$};
		\draw(1,4) node {$\bullet$};
		\draw(2,4) node {$\bullet$};
		\draw(3,4) node {$\bullet$};
		\draw(4,4) node {$\bullet$};
		\draw(0,-2) node {$\bullet$};
		\draw(0,-3) node {$\bullet$};
		\draw(-1,-3) node {$\bullet$};
		\draw(-1,-4) node {$\bullet$};
		\draw(-2,-4) node {$\bullet$};
		\draw(0,-4) node {$\bullet$};
		\draw(-3,-5) node {$\phantom{\iddots}$};
	\end{tikzpicture}
	\begin{tikzpicture}[scale=0.5]
		\justaxis{}
		\draw(4,0) node {$p$};
		\draw(0,5) node {$q$};
		\HMtwo{0}{0}
		\draw(-3,-5) node {$\phantom{\iddots}$};
	\end{tikzpicture}\\
	\caption{Several visual representations of $\Mt.$ Copies of $\Zt$ are represented with $\bullet$ in the middle representation. On the right-hand representation, the groups are merely implied.}
	\label{fig:1}
\end{figure}

In the third part of Figure \ref{fig:1}, we have labeled the $p$-axis (or dimension-axis) and the $q$-axis (or weight axis). We see the ring divided into a \textbf{top cone} consisting of elements of the form $\rho^i\tau^j$ and a \textbf{lower cone} of elements $\frac\theta{\rho^i\tau^j}$. Even this last representation can get messy, and so we will often abbreviate further. For example, we will see later that

$$\H\bullet\bullet(\Gr_2(\R^{4,1}))=\Mt\oplus\M11\oplus( \M21)^{\oplus 2}\oplus \M31\oplus \M42$$
and visualizing this free module will often be easier if we only worry about the generators of this free $\Mt$-module, as in Figure \ref{fig:shorthand}.
\begin{figure}[H]
	\begin{tikzpicture}[scale=0.4]
		\axisname{}
		\HMtwo{0}{0}
		\HMtwo{1.2}{0.9}
		\HMtwo{2}{1}
		\HMtwo{2.2}{0.9}
		\HMtwo{3}{1}
		\HMtwo{4.2}{1.9}
	\end{tikzpicture}
	\!\!\!\!\!
	\begin{tikzpicture}[scale=.6]
		\def\pmax{ 4 }
		\def\qmax{ 3 }
		\draw (0.5*\pmax,-1.5) node[anchor=south] {$   $};
		\def\points{(0.1, 0.1), (1.1, 1.1), (2.2, 1.1), (2.1, 1.2), (3.1, 1.1), (4.1, 2.1)}
		\draw[step=1cm,lightgray,very thin] (0,0) grid (\pmax+0.5,\qmax+.5);
		\draw[thick,->] (0,-0.5) -- (0,.5+\qmax);
		\draw[thick,->] (-0.5,0) -- (.5+\pmax,0);
		\foreach \x in {1,...,\pmax}
			\draw (\x cm,1pt) -- (\x cm,-1pt) node[anchor=north] {$\x$};
		\foreach \y in {1,...,\qmax}
			\draw (1pt,\y cm) -- (-1pt,\y cm) node[anchor=east] {$\y$};
		\foreach \pq in \points
			\draw \pq [fill] circle (.04);
	\end{tikzpicture}
	\quad
	\begin{tikzpicture}[scale=.6]
		\def\pmax{ 4 }
		\def\qmax{ 3 }
		\draw (0.5*\pmax,-1.5) node[anchor=south] {$   $};
		\draw[step=1cm,lightgray,very thin] (0,0) grid (\pmax+0.5,\qmax+.5);
		\draw[thick,->] (0,-0.5) -- (0,.5+\qmax);
		\draw[thick,->] (-0.5,0) -- (.5+\pmax,0);
		\foreach \x in {1,...,\pmax}
			\draw (\x cm,1pt) -- (\x cm,-1pt);
		\foreach \y in {1,...,\qmax}
			\draw (1pt,\y cm) -- (-1pt,\y cm);
		\draw (0,0) node[above right] {$1$};
		\draw (1,1) node[above right] {$1$};
		\draw (2,1) node[above right] {$2$};
		\draw (3,1) node[above right] {$1$};
		\draw (4,2) node[above right] {$1$};
	\end{tikzpicture}\\
	\caption{Several visual representations of $\H\bullet\bullet(\Gr_2(\R^{4,1}))$. The last of these is called a \textbf{rank chart}.}
	\label{fig:shorthand}
\end{figure}

\subsection{Warning}
	The shorthand in the second and third diagrams of Figure \ref{fig:shorthand} can be a mercy, but also runs the risk of deception, as certain bidegrees appear ``empty'' but aren't. For example, while it is clear from the leftmost diagram (with some squinting) that $\H22=(\Zt)^4$, this is not clear at a glance from the other two; we must remember to imagine the upper and lower cones.

\subsection{A forgetful long exact sequence}
The following theorem appearing in \cite{araki} relates this cohomology theory to singular cohomology (with $\Zt$ coefficients). Denote the equivariant Eilenberg-MacLane space representing $\H pq$ by $K(\Zt,p,q)$.

\begin{thm} 
\label{rholes}	
For fixed $q$, there is a long exact sequence 
	\[\dots\to\H pq(X)\xrightarrow{\cdot\rho}\H{p+1}{q+1}(X)\xrightarrow{\psi}H^{p+1}_{\text{sing}}(X)\to\H{p+1}q(X)\xrightarrow{\cdot\rho}\dots\]
	where $\psi$ is the \textbf{forgetful map} $[X,K(\underline{\Zt},p,q)]_{C_2-\Top}\to [X,K(\Zt,p)]_{\Top}$
\end{thm}

It is clear that $\psi:\Mt=\H\bullet\bullet(\pt)\to H_{\text{sing}}^\bullet(\pt)$ takes $\rho$ to 0. Notice this implies that $\psi(\theta)=0$, since $\theta$ is $\rho$-divisible. We will also make use of the fact that $\psi(\tau)=1$. 
 These facts have a nice geometric interpretation using the Dold-Thom model of Eilenberg-MacLane spaces. We omit this interpretation, but geometric models for $\rho$, $\tau$ and $\theta$ can be found in Proposition 4.5 of \cite{clover}.

\theoremstyle{definition}

\begin{defn}
	\label{def:repcell}
	A \textbf{representation disc} $D^{p,q}=D(\R^{p,q})$ is the closed unit disc in a representation, and a \textbf{representation cell} $e^{p,q}$ is its interior. A space which can be built from representation cells by the usual gluing diagrams (now with equivariant attaching maps out of $\partial D^{p,q}$) is said to have a \textbf{representation cell structure}.
\end{defn}

We will make use of Kronholm's\footnote{This theorem is true, however the proof given in \cite{kronholm} is problematic. For another proof see in \cite{buddies}.} freeness theorem.

\begin{thm}[Kronholm]
	\label{freeness}
	If a (locally finite, finite-dimensional) $\Ct$-space $X$ has a representation cell structure then it has free cohomology:
	\[\H\bullet\bullet(X)=\bigoplus_i\Sigma^{a_i,b_i}\Mt=\bigoplus_i\rH\bullet\bullet(\S{a_i}{b_i})\]
	for some bidegrees $\{(a_i,b_i)\}_i$.
\end{thm}

The bidegrees $(a_i,b_i)$ need not coincide with those of the representation cells used to build $X$, as the weights $b_i$ may differ. While the cohomologies of many families of Grassmannians remain unknown, we next present the known results.

\subsection{Formulas}

Kronholm calculated the cohomology of the various projective spaces $\Gr_1(\R^{p,q})=\mathbb{P}(\R^{p,q})$. Taking $p\ge 2q$, 

\[\H\bullet\bullet(\Gr_1\R^{p,q})
=\Mt\oplus\bigoplus_{i=1}^{q-1}(\M {2i-1}i\oplus \M {2i}i)\oplus\bigoplus_{j=2q-1}^{p-1}\M jq.\]
For example, $\H\bullet\bullet(\mathbb{P}(\R^{11,4}))$ is represented below. 
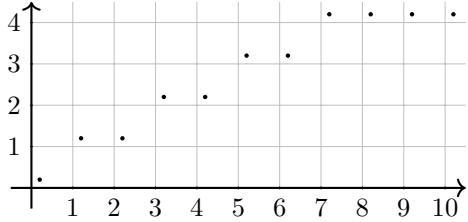
\begin{center}
	\begin{tikzpicture}[scale=0.55]
		\def\pmax{ 10 }
		\def\qmax{ 4 }
		\draw (0.5*\pmax,-1.5) node[anchor=south] {$   $};
		\def\points{(0.2, 0.2), (1.2, 1.2), (2.2, 1.2), (3.2, 2.2), (4.2, 2.2), (5.2, 3.2), (6.2, 3.2), (7.2, 4.2), (8.2, 4.2), (9.2, 4.2), (10.2, 4.2)}
		\draw[step=1cm,lightgray,very thin] (0,0) grid (\pmax+0.5,\qmax+.5);
		\draw[thick,->] (0,-0.5) -- (0,.5+\qmax);
		\draw[thick,->] (-0.5,0) -- (.5+\pmax,0);
		\foreach \x in {1,...,\pmax}
			\draw (\x cm,1pt) -- (\x cm,-1pt) node[anchor=north] {$\x$};
		\foreach \y in {1,...,\qmax}
			\draw (1pt,\y cm) -- (-1pt,\y cm) node[anchor=east] {$\y$};
		\foreach \pq in \points
			\draw \pq [fill] circle (.04);
	\end{tikzpicture}
\end{center}

In Section \ref{knone} we prove a theorem for the family $\Gr_k(\R^{n,1})$, extending results of \cite{kronholm} for $\Gr_2(\R^{n,1})$. As in \cite{dugger}, define the $\Mt$-rank of a free $\Mt$-module $M$ by letting $I=\ker(\Mt\to\Zt)$ and set 
$$\rank^{p,q}_{\Mt}(M)=\dim_{\Zt}(\sfrac M{IM})^{p,q}.$$

Given non-negative, weakly-increasing numbers $\lambda_1, \lambda_2, \ldots, \lambda_k$, define the \textbf{trace} of this collection to be $\#\{i:\lambda_i\ge k-i+1\}=t$. If a $\lambda_i$ is interpreted as a Young diagram, its trace represents the number of boxes on the main diagonal. See Table \ref{table:tracetable} for examples. Let $\part(p,k,m,t)$ denote the number of partitions of $p$ into $k$ numbers $\lambda_i\le m$ having trace $t$. Using this definition, we state the following theorem.

\begin{thm}
	\label{kn1thm}
	\[\rank_{\Mt}^{p,q}\H \bullet\bullet(\Gr_k(\R^{n,1}))=\part(p,k,n-k,q).\]
\end{thm}

In words, the free generators of $\H pq(\Gr_k(\R^{n,1}))$ having degree $(p,q)$ are counted by trace-$q$ Young diagrams of $p$ boxes fitting inside of a $k$-by-$(n-k)$ box. This formula lets us calculate cohomologies like that of $\Gr_{4}(\R^{9,1})$, shown in Figure \ref{fig:491}.

\begin{figure}[H]
	\begin{tikzpicture}[scale=0.6]
		\def\pmax{20}
		\def\qmax{4}
		\draw(21,0) node {$p$};
		\draw(0,5) node {$q$};
		\draw[step=1cm,lightgray,very thin] (0,0) grid (\pmax+0.5,\qmax+.5);
		\draw[thick,->] (0,-0.5) -- (0,.5+\qmax);
		\draw[thick,->] (-0.5,0) -- (.5+\pmax,0);
		\draw (5.3,0) node[below] {\rotatebox{-90}{$5$}};
		\draw (10.3,0) node[below] {\rotatebox{-90}{$10$}};
		\draw (15.3,0) node[below] {\rotatebox{-90}{$15$}};
		\draw (20.3,0) node[below] {\rotatebox{-90}{$20$}};
		\draw(-.4,2.2) node {$2$};
		\draw(-.4,4.2) node {$4$};
		\draw (19.16,4.3) node {$1$};
		\draw (18.16,4.3) node {$1$};
		\draw (16.16,3.3) node {$4$};
		\draw (20.16,4.3) node {$1$};
		\draw (1.16,1.3) node {$1$};
		\draw (8.3,2.3) node {$10$};
		\draw (9.16,3.3) node {$1$};
		\draw (14.16,2.3) node {$1$};
		\draw (3.16,1.3) node {$3$};
		\draw (2.16,1.3) node {$2$};
		\draw (5.16,1.3) node {$4$};
		\draw (10.16,3.3) node {$2$};
		\draw (7.16,2.3) node {$7$};
		\draw (12.16,2.3) node {$5$};
		\draw (13.16,3.3) node {$7$};
		\draw (7.16,1.3) node {$2$};
		\draw (4.16,1.3) node {$4$};
		\draw (16.16,4.3) node {$1$};
		\draw (9.3,2.3) node {$10$};
		\draw (6.16,1.3) node {$3$};
		\draw (11.16,3.3) node {$4$};
		\draw (14.16,3.3) node {$7$};
		\draw (6.16,2.3) node {$5$};
		\draw (12.16,3.3) node {$6$};
		\draw (13.16,2.3) node {$2$};
		\draw (15.16,3.3) node {$6$};
		\draw (4.16,2.3) node {$1$};
		\draw (11.16,2.3) node {$7$};
		\draw (17.16,3.3) node {$2$};
		\draw (8.16,1.3) node {$1$};
		\draw (0.16,0.3) node {$1$};
		\draw (17.16,4.3) node {$1$};
		\draw (18.16,3.3) node {$1$};
		\draw (5.16,2.3) node {$2$};
		\draw (10.3,2.3) node {$10$};
	\end{tikzpicture}
	\caption{Rank chart for $\H\bullet\bullet(\Gr_4(\R^{9,1}))$.}
	\label{fig:491}
\end{figure}

For example, the 5 in bidegree $(6,2)$ says that $\rank_{\Mt}^{6,2}(\H\bullet\bullet(\Gr_4(\R^{9,1})))=5$ which is counted by $\part(6,4,8-4,2)$, the number of partitions of 6 into 4 numbers each at most $4$, with trace $t=\#\{i:\lambda_i\ge k-i+1\}=2$. These are the starred entries in Table \ref{table:tracetable}.
\begin{table}[h]
	\begin{tabular}{cc|c|c}
		&Partition of 6 & Trace & Young Diagram\\
		\hline
		*&0+0+2+4 &2& \scalebox{0.5}{\young({\,}\diagup,\diagup{\,}{\,}{\,})} {\color{white}$\ds\int$}\\
		*&0+0+3+3 &2& \scalebox{0.5}{\young({\,}\diagup{\,},\diagup{\,}{\,})}{\color{white}$\ds\int$}\\
		&0+1+1+4 &1& \scalebox{0.5}{\young({\,},{\,},\diagup{\,}{\,}{\,})}{\color{white}$\ds\int$}\\
		*&0+1+2+3 &2& \scalebox{0.5}{\young({\,},{\,}\diagup,\diagup{\,}{\,})}{\color{white}$\ds\int$}\\
		*&0+2+2+2 &2& \scalebox{0.5}{\young({\,}{\,},{\,}\diagup,\diagup{\,})}{\color{white}$\ds\int$}\\
		&1+1+1+3 &1& \scalebox{0.5}{\young({\,},{\,},{\,},\diagup{\,}{\,})}{\color{white}$\ds\int$}\\
		*&1+1+2+2 &2& \scalebox{0.5}{\young({\,},{\,},{\,}\diagup,\diagup{\,})}{\color{white}$\ds\int$}\\
	\end{tabular}
	\caption{Partitions of $6$ into $4$ nonnegative integers not exceeding 4, and their traces.}
	\label{table:tracetable}
\end{table}

\subsection{Comment}
\label{foreshadow}
The reader may have noticed that the rows of the rank table in Figure \ref{fig:491} are palindromes. There is a simple combinatorial reason for this, which we will give in Section \ref{upside}.\\

In Section \ref{twoeighttwo} we will also prove the following:

\begin{thm}
	\label{uglyformula}
	The cohomology of $\Gr_2(\R^{n,2})$ with $n\ge 6$, is given by
	\begin{align*}
		\H\bullet\bullet(\Gr_2(\R^{n,2}))&=
		\Mt\oplus \M11\oplus \M21
	 	 \oplus\M22\oplus (\M32)^{\oplus 2}\oplus (\M42)^{\oplus 3}\\
		 &\oplus\bigoplus_{p=5}^{n-2}(\M p2)^{\oplus 2}\oplus\M{n-1}2\\	 
		 &\oplus \M53\oplus\bigoplus_{p=6}^n(\M p3)^{\oplus 2}
		 \oplus \M{n+1}3\\
		 &\oplus \bigoplus_{p=8}^{n+1}(\M p4)^{\oplus\lceil\frac{p-7}2\rceil}\\
		 &\oplus \bigoplus_{p=n+2}^{2n-4}(\M p4)^{\oplus (n-1-\lceil\frac{p}2\rceil)}
	\end{align*}
\end{thm}

For example $\H\bullet\bullet(\Gr_2(\R^{10,2}))$ is represented in Figure \ref{fig:2ten2}. Note that each line of the formula in Theorem \ref{uglyformula} corresponds to a different circled region. The first is common to all of them (provided $n\ge 6$) and the next two stretch predictably as $n$ grows. The top row is made up of a region where ranks increase left-to-right every two steps, and another in which ranks decrease left-to-right in the same way. For $n\ge 6$ it is convenient to organize the data in this way, but we also calculate these cohomologies for $3\le n<6$ in Section \ref{twontwo}.

\begin{figure}[h]
	\begin{tikzpicture}[scale=0.7]
		\def\pmax{16.2}
		\def\qmax{4}
		\draw[step=1cm,lightgray,very thin] (0,0) grid (\pmax+0.5,\qmax+.5);
		\draw[thick,->] (0,-0.5) -- (0,.5+\qmax);
		\draw[thick,->] (-0.5,0) -- (.5+\pmax,0);
		\draw (7.15,3.3) node {$2$};
		\draw (10.15,4.3) node {$2$};
		\draw (9.15,3.3) node {$2$};
		\draw (14.15,4.3) node {$2$};
		\draw (2.15,1.3) node {$1$};
		\draw (6.15,2.3) node {$2$};
		\draw (9.15,4.3) node {$1$};
		\draw (10.15,3.3) node {$2$};
		\draw (7.15,2.3) node {$2$};
		\draw (16.15,4.3) node {$1$};
		\draw (6.15,3.3) node {$2$};
		\draw (13.15,4.3) node {$2$};
		\draw (2.15,2.3) node {$1$};
		\draw (15.15,4.3) node {$1$};
		\draw (1.15,1.3) node {$1$};
		\draw (3.15,2.3) node {$2$};
		\draw (0.15,0.3) node {$1$};
		\draw (11.15,4.3) node {$2$};
		\draw (8.15,2.3) node {$2$};
		\draw (4.15,2.3) node {$3$};
		\draw (5.15,3.3) node {$1$};
		\draw (8.15,3.3) node {$2$};
		\draw (9.15,2.3) node {$1$};
		\draw (5.15,2.3) node {$2$};
		\draw (11.15,3.3) node {$1$};
		\draw (12.15,4.3) node {$3$};
		\draw (8.15,4.3) node {$1$};
		\draw [blue] plot [smooth cycle] coordinates {(0,0) (4,1) (4.3,3) (1,2.1)};
		\draw [blue] plot [smooth cycle] coordinates {(5.1,2) (9.1,2) (9.1,2.7) (5.1,2.7)};
		\draw [blue] plot [smooth cycle] coordinates {(5.1,3) (11.1,3) (11.1,3.7) (5.1,3.7)};
		\draw [blue] plot [smooth cycle] coordinates {(8.1,4) (11.1,4) (11.1,4.7) (8.1,4.7)};
		\draw [blue] plot [smooth cycle] coordinates {(12.1,4) (16.1,4) (16.1,4.7) (12.1,4.7)};
		\draw(-.4,2.2) node {$2$};
		\draw(-.4,4.2) node {$4$};
		\draw (0.3,0) node[below] {\rotatebox{-90}{$0$}};
		\draw (5.3,0) node[below] {\rotatebox{-90}{$5$}};
		\draw (8.3,0) node[below] {\rotatebox{-90}{$n-2$}};
		\draw (9.3,0) node[below] {\rotatebox{-90}{$n-1$}};
		\draw (10.3,0) node[below] {\rotatebox{-90}{$n$}};
		\draw (11.3,0) node[below] {\rotatebox{-90}{$n+1$}};
		\draw (12.3,0) node[below] {\rotatebox{-90}{$n+2$}};
		\draw (16.3,0) node[below] {\rotatebox{-90}{$2n-4$}};
	\end{tikzpicture}
	\caption{Rank chart for $\H\bullet\bullet(\Gr_2(\R^{n,2}))$ with $n=10$.}
	\label{fig:2ten2}
\end{figure}

Analogous formulas for complex Grassmannians, whose cohomologies look similar but have generators with twice the topological degree and weight, will also appear in Section \ref{complexsection}.

\subsection{Note} 
	It should be remembered that while the rank table in Figure \ref{fig:2ten2} organizes all of the information about a free rank-45 $\Mt$-module much more pleasantly than a list of summands would, it may also leave too much to the imagination. For example, while bidegree $(4,0)$ appears empty, actually $\H40(\Gr_2(\R^{10,2}))=(\Zt)^4$, generated by the $\theta$-multiples of the generators of three distinct copies of $\M42$, and also the $\frac\theta\rho$-multiple of the generator of $\M53$. Likewise $\H23=(\Zt)^4$ is generated by $\tau\cdot1_{\M22}\in\M22$ as well as $\rho\tau\cdot 1_{\M11}\in\M11$, $\tau^2\cdot 1_{\M21}\in\M21$ and $\rho^2\tau \cdot 1_{\Mt}\in \Mt$.

\subsection{Acknowledgements} This work is based on the author's doctoral dissertation. It is the product of many conversations with both his thesis advisor Dan Dugger and with Clover May, who each came to the rescue repeatedly. The author is grateful to both of them, as well as to Kelly Poland for spotting an error, and to the anonymous referee for numerous useful suggestions. This work was partially funded by University of Oregon and Gonzaga University.
\\

\section{Background on the representation-cell structure}

Before we present and prove general results for these cohomologies, we will work a few manageable examples bare-handed, to give the reader a feel for equivariant long exact sequence computations. (Note this is distinct from the spectral sequence approach, which we will also make use of later.)

\subsection{Worked Example I}
\label{we1}
This example serves primarily to demonstrate the phenomenon of the ``Kronholm shift,'' found in \cite{kronholm} and \cite{buddies}.\par

When using a CW structure to calculate the singular cohomology of a space, we can work iteratively on skeleta, attaching one $k$-cell at a time. The cofiber sequence $X_{n-1}\inj X_n\to S^k$ then gives a long exact sequence, and if we know $H^i_{\text{sing}}(X_{n-1})$ and the differential $H^i_{\text{sing}}(X_{n-1})\xrightarrow{d} H^{i+1}_{\text{sing}}(S^k)$, we can (at least over $\Zt$) deduce $H_{\text{sing}}^i(X_n)$. \par

The equivariant situation is analogous: The equivariant cofiber sequence $X_{n-1}\inj X_n\to \S pq$ extends to a Puppe sequence
\[\dots\to\Sigma^{-1,0}\S pq\to X_{n-1}\inj X_n\to \S pq\to\Sigma^{1,0}X_{n-1}\]
yielding a long exact sequence of $\Mt$-module maps in cohomology, including a \textbf{differential} $d:\H\bullet\bullet(X_{n-1})\to \H\bullet\bullet(\Sigma^{-1,0}\S pq)=\H{\bullet+1}\bullet(\S pq)$. It turns out that certain zero differentials in the non-equivariant theory are actually the ``shadows'' of something more interesting in the equivariant theory.\par

Consider $\Gr_1(\R^{3,1})$, whose underlying space is $\Gr_1(\R^3)=\RP^2$. We can build the space from representation cells in two ways (See Figure \ref{fig:twoways}). First, we can begin with a point, attach a non-trivial line segment $e^{1,1}\iso \R^{1,1}$ (thus building $\S11$) and finally attach $e^{2,1}\iso\R^{2,1}$ via a degree-two map from its boundary $\partial\D21=\S11$. Alternatively, we can build this space by attaching a fixed interval $e^{1,0}$ to a point, and then attaching $e^{2,2}$ to this fixed circle with a degree-two map.\par

\begin{figure}[h]
	\begin{tikzpicture}
		\draw[fill=lightgray, opacity=0.5] (1,0) arc (0:360:1);
		\draw (.93, .34) arc (20:90:1);
		\draw[->] (0,1) arc (90:200:1);
		\draw (-.93, -.34) arc (200:270:1);
		\draw[->] (0,-1) arc (270:380:1);
		\draw [dashed] (0,1.5) -- (0,1);
		\draw [dashed] (0,-1.5) -- (0,-1);
		\draw [<->] (-.4,1.35)--(.4,1.35);
		\draw [<->] (-.4,-1.35)--(.4,-1.35);
		\color{red}
		\draw [thick] (0,1) -- (0,-1);
		\fill (1, 0) circle[radius=0.07cm];
		\fill (-1, 0) circle[radius=0.07cm];
		\color{black}
		\draw (0,-2.5) node {
		\xymatrix{\ast=X_0\ar[r]& X_1\ar[r]\ar[d]& X_2\ar[d]\\ &\S11&\S21}
		};
	\end{tikzpicture}
	\qquad\qquad
	\begin{tikzpicture}
		\draw[fill=lightgray, opacity=0.5] (1,0) arc (0:360:1);
		\draw[->] (0,.5) arc (90:260:.5);
		\draw[->] (0,-.5) arc (270:360+80:.5);
		\color{red}
		\draw[thick] (1, 0) arc (0:90:1);
		\draw[->,thick] (0,1) arc (90:180:1);
		\draw[thick] (-1,0) arc (180:270:1);
		\draw[thick, ->] (0,-1) arc (270:360:1);
		\fill (0, 0) circle[radius=0.07cm];
		\draw (0,-1.5) node{};
		\draw (0,1.5) node{};
		\color{black}
		\draw (0,-2.5) node {
		\xymatrix{\ast=X_0\ar[r]& X_1\ar[r]\ar[d]& X_2\ar[d]\\ &\S10&\S22}
		};
	\end{tikzpicture}\\
	\caption{Fixed points in thick red. Note (taking identifications into account) the fixed circle and (single) isolated fixed point appearing in each diagram. Below the two constructions of $\Gr_1(\R^{3,1})$ are their filtration quotients.}
	\label{fig:twoways}
\end{figure}

In the first construction, the cofiber sequence for including the one-skeleton is $\S11\inj\Gr_1(\R^{3,1})\to \S21$. The differential  $d:\rH\bullet\bullet(\S11)\to\rH{\bullet+1}\bullet(\S21)\iso\rH\bullet\bullet\S11$ (depicted on the left of Figure \ref{fig:twodiffs}) must be zero, otherwise the forgetful map would predict a nonzero map $\psi(d)$ in the non-equivariant cellular chain complex. Thus we know relatively easily that 
	$$\H\bullet\bullet(\Gr_1(\R^{3,1}))=\Mt\oplus\M11\oplus \M21.$$
	\begin{figure}[h]
	\begin{tikzpicture}[scale=0.55]
			\axisname{}
			\HMtwo{0}{0}
			\HMtwo{1.1}{1}
			\draw[->] (1.3,1)--(1.8,1) node[below] {$d=0$};
			\color{blue}
			\HMtwo{2}{1}
	\end{tikzpicture}
	\begin{tikzpicture}[scale=0.55]
			\axisname{}
			\HMtwo{0}{0}
			\HMtwo{1}{0}
			\draw[->] (1.2,0)--(1.8,0);
			\draw (2.3,-0.1) node[above] {$d\ne 0$};
			\color{blue}
			\HMtwo{2.1}{2}
	\end{tikzpicture}
	\caption{Differentials from attaching 2-cells.}
	\label{fig:twodiffs}
	\end{figure}
	However in the second construction for the same space, we have the cofiber sequence $\S10\inj\Gr_1(\R^{3,1})\to\S22$. In this case the differential
	\[d:\H\bullet\bullet\S10=\M10\to\H\bullet\bullet\S22=\M22\]
	\emph{cannot} be zero, or we would have two conflicting answers. Rather, $d(1_{\M10})=\theta1_{\M22}$, and we have an extension problem with $\ker(d)$ and $\cok(d)$. While we already know the answer in this case, this problem is resolved generally by \cite{kronholm} and \cite{buddies}. Heuristically, the differential into the lower cone causes $\M10$ to `shift up' to become a $\M11$ while $\M22$ `shifts down' to a $\M21$, replicating the cohomology we expect from the first construction.\par
	
\subsection{Note}
		This phenomenon of nonzero differentials into a lower cone causing shifted weights in the free $\Mt$ generators is called a \textbf{Kronholm shift}. In its simplest version, where just one $\Mt$ maps into the lower cone of another, the source $\Mt$ shifts up by the difference in fixed set dimension of the two free generators, and the target $\Mt$ shifts down by the same amount. A more general formula for shifts when an arbitrary number of $\Mt$s have nonzero-differentials to a lower cone appears in \cite{buddies}.\par

	This trick of deducing properties of unknown differentials in one representation-cell construction (see Definition \ref{def:repcell}) by leveraging what is known about another construction continues to be a useful strategy as we move to larger Grassmannians.

\subsection{Schubert cells}
\label{we2}

Non-equivariantly, $\Gr_k(\R^n)$ can be given a cell structure indexed by Young diagrams fitting inside a $k$-by-$(n-k)$ rectangle. See e.g. \cite{manivel}. For example, $\Gr_2(\R^5)$ can be built with cells indexed by diagrams fitting into {\tiny$\yng(3,3)$} as follows:

\begin{figure}[H]
	\begin{tikzpicture}[scale=0.9]
		\draw (2,0) node(z) {$(0,0)$};
		\draw (2,1) node(o) {$(0,1)$};
		\draw (1,2) node(oo) {$(1,1)$};
		\draw (3,2) node(t) {$(0,2)$};
		\draw (2,3) node(ot) {$(1,2)$};
		\draw (4,3) node(th) {$(0,3)$};
		\draw (1,4) node(tt) {$(2,2)$};
		\draw (3,4) node(oth) {$(1,3)$};
		\draw (2,5) node(tth) {$(2,3)$};
		\draw (2,6) node(thth) {$(3,3)$};
		\draw [-, ultra thin] (z) -- (o);
		\draw [-, ultra thin] (o) -- (oo);
		\draw [-, ultra thin] (o) -- (t);
		\draw [-, ultra thin] (oo) -- (ot);
		\draw [-, ultra thin] (t) -- (ot);
		\draw [-, ultra thin] (t) -- (th);
		\draw [-, ultra thin] (ot) -- (tt);
		\draw [-, ultra thin] (ot) -- (oth);
		\draw [-, ultra thin] (th) -- (oth);
		\draw [-, ultra thin] (tt) -- (tth);
		\draw [-, ultra thin] (oth) -- (tth);
		\draw [-, ultra thin] (tth) -- (thth);
	\end{tikzpicture}
	\quad
	\begin{tikzpicture}[scale=0.9]
		\draw (2,0) node(z) {$\emptyset$};
		\draw (2,1) node(o) {\tiny$\yng(1)$};
		\draw (1,2) node(oo) {\tiny$\yng(1,1)$};
		\draw (3,2) node(t) {\tiny$\yng(2)$};
		\draw (2,3) node(ot) {\tiny$\yng(1,2)$};
		\draw (4,3) node(th) {\tiny$\yng(3)$};
		\draw (1,4) node(tt) {\tiny$\yng(2,2)$};
		\draw (3,4) node(oth) {\tiny$\yng(1,3)$};
		\draw (2,5) node(tth) {\tiny$\yng(2,3)$};
		\draw (2,6) node(thth) {\tiny$\yng(3,3)$};
		\draw [-, ultra thin] (z) -- (o);
		\draw [-, ultra thin] (o) -- (oo);
		\draw [-, ultra thin] (o) -- (t);
		\draw [-, ultra thin] (oo) -- (ot);
		\draw [-, ultra thin] (t) -- (ot);
		\draw [-, ultra thin] (t) -- (th);
		\draw [-, ultra thin] (ot) -- (tt);
		\draw [-, ultra thin] (ot) -- (oth);
		\draw [-, ultra thin] (th) -- (oth);
		\draw [-, ultra thin] (tt) -- (tth);
		\draw [-, ultra thin] (oth) -- (tth);
		\draw [-, ultra thin] (tth) -- (thth);
	\end{tikzpicture}
	\qquad
	\begin{tikzpicture}[scale=0.9]
		\draw (2,0) node(z) {$[1,2]$};
		\draw (2,1) node(o) {$[1,3]$};
		\draw (1,2) node(oo) {$[2,3]$};
		\draw (3,2) node(t) {$[1,4]$};
		\draw (2,3) node(ot) {$[2,4]$};
		\draw (4,3) node(th) {$[1,5]$};
		\draw (1,4) node(tt) {$[3,4]$};
		\draw (3,4) node(oth) {$[2,5]$};
		\draw (2,5) node(tth) {$[3,5]$};
		\draw (2,6) node(thth) {$[4,5]$};
		\draw [-, ultra thin] (z) -- (o);
		\draw [-, ultra thin] (o) -- (oo);
		\draw [-, ultra thin] (o) -- (t);
		\draw [-, ultra thin] (oo) -- (ot);
		\draw [-, ultra thin] (t) -- (ot);
		\draw [-, ultra thin] (t) -- (th);
		\draw [-, ultra thin] (ot) -- (tt);
		\draw [-, ultra thin] (ot) -- (oth);
		\draw [-, ultra thin] (th) -- (oth);
		\draw [-, ultra thin] (tt) -- (tth);
		\draw [-, ultra thin] (oth) -- (tth);
		\draw [-, ultra thin] (tth) -- (thth);
	\end{tikzpicture}
	\caption{Partition tuples, Young diagrams, and jump sequences.}
	\label{fig:threeversions}
\end{figure}
To each Young diagram written as a weakly ascending tuple $\lambda$ (for example \scalebox{0.4}{\yng(1,3)} corresponds to $\lambda=(1,3)$) we can write a strictly ascending tuple $\underline{j}=[\lambda_i+i]_i$ called the \textbf{jump sequence}. The diagram \scalebox{0.4}{\yng(1,3)} has jump sequence $[1+1,3+2]$. These are the symbols on the right-hand side of Figure \ref{fig:threeversions}.\par 
These symbols index the cells of the Grassmannian as follows. We can think of a $k$-plane in $\R^n$ as the rowspace of a $k$-by-$n$ matrix, and without loss of generality, this matrix can be written in a \textbf{canonical form} (lower triangular reduced row echelon form) so that each row's last nonzero entry is a 1, which then clears the column below it. Order these rows by the position of their last nonzero entry. For example:
\[\rs\begin{bmatrix}2&-2&10&12&2\\21&3&15&18&3\end{bmatrix}=\rs\begin{bmatrix}3&1&0&0&0\\4&0&5&6&1\end{bmatrix}:=V_{\scalebox{0.5}{\young(3,456)}}.\]
In this way, every point in the Grassmannian can be sorted into a unique family, these families indexed by jump sequences which give the locations of these $1$s in their canonical representations. These families are related. Consider for example the open set containing the four-parameter family of all planes of the form $V_{\scalebox{0.6}{\young(w,xyz)}}$ (abbreviating rowspace with rs)
\[\Omega_{[2,5]}=\Omega_{\scalebox{0.4}{\yng(1,3)}}:=\left\{rs\begin{bmatrix}w&1&0&0&0\\x&0&y&z&1\end{bmatrix} \,\,:\,\,w,x,y,z\in \R\right\}\subset \Gr_2(\R^5).\]

 Since
\[\lim_{c\to\infty}\text{rs}\begin{bmatrix}x&1&0&0&0\\cy&0&cz&c&1\end{bmatrix}=\text{rs}\begin{bmatrix}x&1&0&0&0\\y&0&z&1&0\end{bmatrix}\in\Omega_{\scalebox{.4}{\yng(1,2)}}\]
and
\begin{align*}
	\lim_{c\to\infty}\text{rs}\begin{bmatrix}c&1&0&0&0\\-cx&0&y&z&1\end{bmatrix}
	=\lim_{c\to\infty}\text{rs}&\begin{bmatrix}c&1&0&0&0\\0&x&y&z&1\end{bmatrix}\\
	=\text{rs}&\begin{bmatrix}1&0&0&0&0\\0&x&y&z&1\end{bmatrix}\in\Omega_{\scalebox{.4}{\yng(3)}},
\end{align*}

\noindent
we have that the closure $X_{[2,5]}:=\overline{\Omega_{[2,5]}}$ contains $ \Omega_{[2,4]}$ and also $X_{[2,5]}$ contains $\Omega_{[1,5]}$, or in Young diagrams, $\Omega_{\scalebox{0.4}{\yng(3)}}\subset X_{\scalebox{0.4}{\yng(3)}}\subset X_{\scalebox{0.4}{\yng(1,3)}}$ and $\Omega_{\scalebox{0.4}{\yng(1,2)}}\subset X_{\scalebox{0.4}{\yng(1,2)}}\subset X_{\scalebox{0.4}{\yng(1,3)}}$. The sets $\Omega_{\underline{j}}$ indexed by jump sequences (or equivalently by Young diagrams) are called \textbf{Schubert cells}, and their closures \textbf{Schubert varieties}. We have an obvious notion of containment for Young diagrams, to which corresponds a notion of dominance in jump sequences. We say that a jump sequence \underline{$j$} \textbf{dominates} another jump sequence \underline{$k$}, denoted $\underline{k}\prec \underline{j}$, if each $k_i\le j_i$. Containment between Schubert varieties corresponds to containment between their indexing Young diagrams or equivalently, to dominating jump sequences. For further details, see Section 3.2 of \cite{manivel}.

\begin{figure}[h!]
	\begin{tikzpicture}[scale=1.1]
		\draw (3,0) node(z) {$\left\{\text{rs}\left[{1\atop 0}\,{0\atop 1}\,{0\atop 0}\,{0\atop 0}\,{0\atop 0}\right]\right\}$};
		\draw (3,1) node(o) {$\left\{\text{rs}\left[{1\atop 0}\,{0\atop \ast}\,{0\atop 1}\,{0\atop 0}\,{0\atop 0}\right]\right\}$};
		\draw (1.5,2) node(oo) {$\left\{\text{rs}\left[{\ast\atop \ast}\,{1\atop 0}\,{0\atop 1}\,{0\atop 0}\,{0\atop 0}\right]\right\}$};
		\draw (4.5,2) node(t) {$\left\{\text{rs}\left[{1\atop 0}\,{0\atop \ast}\,{0\atop \ast}\,{0\atop 1}\,{0\atop 0}\right]\right\}$};
		\draw (3,3) node(ot) {$\left\{\text{rs}\left[{\ast\atop \ast}\,{1\atop 0}\,{0\atop \ast}\,{0\atop 1}\,{0\atop 0}\right]\right\}$};
		\draw (6,3) node(th) {$\left\{\text{rs}\left[{1\atop 0}\,{0\atop \ast}\,{0\atop \ast}\,{0\atop \ast}\,{0\atop 1}\right]\right\}$};
		\draw (1.5,4) node(tt) {$\left\{\text{rs}\left[{\ast\atop \ast}\,{\ast\atop \ast}\,{1\atop 0}\,{0\atop 1}\,{0\atop 0}\right]\right\}$};
		\draw (4.5,4) node(oth) {$\left\{\text{rs}\left[{\ast\atop \ast}\,{1\atop 0}\,{0\atop \ast}\,{0\atop \ast}\,{0\atop 1}\right]\right\}$};
		\draw (3,5) node(tth) {$\left\{\text{rs}\left[{\ast\atop \ast}\,{\ast\atop \ast}\,{1\atop 0}\,{0\atop \ast}\,{0\atop 1}\right]\right\}$};
		\draw (3,6) node(thth) {$\left\{\text{rs}\left[{\ast\atop \ast}\,{\ast\atop \ast}\,{\ast\atop \ast}\,{1\atop 0}\,{0\atop 1}\right]\right\}$};
		\draw [-, ultra thin] (z) -- (o);
		\draw [-, ultra thin] (o) -- (oo);
		\draw [-, ultra thin] (o) -- (t);
		\draw [-, ultra thin] (oo) -- (ot);
		\draw [-, ultra thin] (t) -- (ot);
		\draw [-, ultra thin] (t) -- (th);
		\draw [-, ultra thin] (ot) -- (tt);
		\draw [-, ultra thin] (ot) -- (oth);
		\draw [-, ultra thin] (th) -- (oth);
		\draw [-, ultra thin] (tt) -- (tth);
		\draw [-, ultra thin] (oth) -- (tth);
		\draw [-, ultra thin] (tth) -- (thth);
	\end{tikzpicture}
	\begin{tikzpicture}[scale=1.1]
		\draw (2,0) node(z) {$\Omega_{[1,2]}=\pt$};
		\draw (2,1) node(o) {$\Omega_{[1,3]}$};
		\draw (1,2) node(oo) {$\Omega_{[2,3]}$};
		\draw (3,2) node(t) {$\Omega_{[1,4]}$};
		\draw (2,3) node(ot) {$\Omega_{[2,4]}$};
		\draw (4,3) node(th) {$\Omega_{[1,5]}$};
		\draw (1,4) node(tt) {$\Omega_{[3,4]}$};
		\draw (3,4) node(oth) {$\Omega_{[2,5]}$};
		\draw (2,5) node(tth) {$\Omega_{[3,5]}$};
		\draw (2,6) node(thth) {$\Omega_{[4,5]}$};
		\draw [-, ultra thin] (z) -- (o);
		\draw [-, ultra thin] (o) -- (oo);
		\draw [-, ultra thin] (o) -- (t);
		\draw [-, ultra thin] (oo) -- (ot);
		\draw [-, ultra thin] (t) -- (ot);
		\draw [-, ultra thin] (t) -- (th);
		\draw [-, ultra thin] (ot) -- (tt);
		\draw [-, ultra thin] (ot) -- (oth);
		\draw [-, ultra thin] (th) -- (oth);
		\draw [-, ultra thin] (tt) -- (tth);
		\draw [-, ultra thin] (oth) -- (tth);
		\draw [-, ultra thin] (tth) -- (thth);
	\end{tikzpicture}
	\caption{Free variables denoted by $\ast$.}
\end{figure}

In this way, Young diagrams index a CW structure for the Grassmannian, each $\scalebox{.5}{\yng(1)}$ in a diagram corresponding to a degree of freedom, and hence the number of boxes equals the dimension of the cell attached at that stage of the construction. (For example, $\Omega_{[2,5]}=\Omega_{\scalebox{0.3}{\yng(1,3)}}\iso  e^4$.)

\newpage

It is an important classical fact that if we work over $\Zt$, the attaching maps given by this CW construction yield only zero differentials in the chain complex, and so (for example) we have singular cohomology
\[H^i_{\text{sing}}(\Gr_2(\R^5);\Zt)=\begin{cases}
	\Zt=\langle[\scalebox{.3}{\yng(3,3)}]\rangle & i=6\\
	\Zt=\langle[\scalebox{.3}{\yng(2,3)}]\rangle & i=5\\
	(\Zt)^2=\langle[\scalebox{.3}{\yng(2,2)}],[\scalebox{.3}{\yng(1,3)}]\rangle & i=4\\
	(\Zt)^2=\langle[\scalebox{.3}{\yng(1,2)}],[\scalebox{.3}{\yng(3)}]\rangle & i=3\\
	(\Zt)^2=\langle[\scalebox{.3}{\yng(1,1)}],[\scalebox{.3}{\yng(2)}]\rangle & i=2\\
	\Zt=\langle[\scalebox{.3}{\yng(1)}]\rangle & i=1\\
	\Zt=\langle[\ast]\rangle & i=0.
\end{cases}
\]
In this notation, the cocycle $\scalebox{.3}{\yng(1,2)}$ is the Kronecker dual to $\Omega_{\scalebox{.3}{\yng(1,2)}}$, that is, it evaluates to 1 on $\Omega_{\scalebox{.3}{\yng(1,2)}}$ and zero on other cells. More generally, cohomology elements are denoted by the Young diagrams of the Schubert cells to which they are dual. This preserves the at-a-glance dimension property.

There is also an equivariant version of this story, which we explain next.

\subsection{Worked Example II}

 Suppose we are interested in $\Gr_2(\R^{5,2})$. If we interpret this as $\Gr_2(\R\triv\oplus\R\sgn\oplus\R\triv\oplus\R\sgn\oplus\R\triv)$ or $\Gr_2(\R^{+-+-+})$ for short, then $\Omega_{[2,5]}$ can be seen to be a representation cell. The action of $\Ct$ on this 4-cell, as in the seventh chapter of \cite{FL}, is given by \label{equivschub}
\[\rs\left[{w\atop x}\,{1\atop 0}\,{0\atop y}\,{0\atop z}\,{0\atop 1}\right]\mapsto
\rs\left[{w\atop x}\,{-1\atop 0}\,{0\atop y}\,{0\atop -z}\,{0\atop 1}\right]
=\rs\left[{-w\atop x}\,{1\atop 0}\,{0\atop y}\,{0\atop -z}\,{0\atop 1}\right]
\]
and so $\Omega_{[2,5]}(\R^{+-+-+})\iso e^{4,2}$ is a representation cell. It is pleasant to write this ${\scriptsize\young(-,++-)}$, as we can see topological dimension and weight at a glance from the number of boxes and minus signs, respectively. Analogous considerations now give a representation cell construction for the space.

\begin{figure}[h!]
	\begin{tikzpicture}[scale=0.8]
		\draw (2,0) node(z) {$\emptyset$};
		\draw (2,1) node(o) {\scalebox{0.7}{$\young(-)$}};
		\draw (1,2) node(oo) {\scalebox{0.7}{$\young(-,+)$}};
		\draw (3,2) node(t) {\scalebox{0.7}{$\young(+-)$}};
		\draw (2,3) node(ot) {\scalebox{0.7}{$\young(-,--)$}};
		\draw (4,3) node(th) {\scalebox{0.7}{$\young(-+-)$}};
		\draw (0.8,4.2) node(tt) {\scalebox{0.7}{$\young(+-,-+)$}};
		\draw (3.2,4.2) node(oth) {\scalebox{0.7}{$\young(-,++-)$}};
		\draw (2,5.4) node(tth) {\scalebox{0.7}{$\young(+-,-+-)$}};
		\draw (2,6.6) node(thth) {\scalebox{0.7}{$\young(-+-,+-+)$}};
		\draw [-, ultra thin] (z) -- (o);
		\draw [-, ultra thin] (o) -- (oo);
		\draw [-, ultra thin] (o) -- (t);
		\draw [-, ultra thin] (oo) -- (ot);
		\draw [-, ultra thin] (t) -- (ot);
		\draw [-, ultra thin] (t) -- (th);
		\draw [-, ultra thin] (ot) -- (tt);
		\draw [-, ultra thin] (ot) -- (oth);
		\draw [-, ultra thin] (th) -- (oth);
		\draw [-, ultra thin] (tt) -- (tth);
		\draw [-, ultra thin] (oth) -- (tth);
		\draw [-, ultra thin] (tth) -- (thth);
	\end{tikzpicture}
	\caption{One representation-cell structure for $\Gr_2(\R^{5,2})$, produced by the choice $\R^{5,2}\iso \R^{+-+-+}$.}
\end{figure}

Once an ordered decomposition of the representation as a direct sum of irreducibles is chosen, the process of assigning weights to Schubert cells can easily be automated. Essentially, to find the weight of a cell, one needs to count the free variables in the associated matrix which are inverted by the group action, once the matrix is returned to canonical form. This amounts to counting the minus signs in a matrix like the third one appearing in Figure \ref{fig:weightcompute}'s example. 

\begin{figure}[h]
	\[
	\begin{matrix}
		\begin{matrix}
			\begin{matrix}
			\quad+\,\,\,-\,\,+\,\,\,-\,\,+\,\,\,-\,\,-\,\,\,+
			\end{matrix}\\
			\text{rs}\left[\begin{matrix}
					\ast&\ast&1&0&0&0&0&0\\
					\ast&\ast&0&1&0&0&0&0\\
					\ast&\ast&0&0&\ast&\ast&1&0\\
					\ast&\ast&0&0&\ast&\ast&0&1\\
			\end{matrix}\right]
		\end{matrix}&
		\mapsto\,\,
		\text{rs}\begin{matrix}
			\begin{matrix}
			+\,&-\,&+\,&-\,&+\,&-\,&-\,&+
			\end{matrix}\\
			\left[\begin{matrix}
					\ast&-\ast&1&0&0&0&0&0\\
					\ast&-\ast&0&-1&0&0&0&0\\
					\ast&-\ast&0&0&\ast&-\ast&-1&0\\
					\ast&-\ast&0&0&\ast&-\ast&0&1\\
			\end{matrix}\right]
		\end{matrix}\\
		&=\text{rs}
		\begin{matrix}
			\begin{matrix}
			\phantom{.}\\
			\end{matrix}\\
			\left[\begin{matrix}
					\ast&-\ast&1&0&0&0&0&0\\
					-\ast&\ast&0&1&0&0&0&0\\
					-\ast&\ast&0&0&-\ast&\ast&1&0\\
					\ast&-\ast&0&0&\ast&-\ast&0&1\\
			\end{matrix}\right]
		\end{matrix}
	\end{matrix}
	\]
	\caption{The number of minus signs in the canonical matrix gives the weight of the Schubert cell with jump sequence $[3,4,7,8]$ in the construction associated to $\R^{+-+-+--+}$.}
	\label{fig:weightcompute}
\end{figure}

While it is preferable to automate this computation, a formula for counting these minus signs can be given for the ordered decomposition $\R^{s(1)}\oplus\R^{s(2)}\oplus\dots\oplus \R^{s(n)}$ with $s:[1,n]\to\{+,-\}$ by letting $\lambda$ have jump sequence $\underline{j}$ and using the reverse Kronecker delta $\widehat \delta_{i,j}=1-\delta_{i,j}$,

\[w\left(\Omega_\lambda(\R^{s(1)s(2)\dots s(n)})\right)=\sum_{k\in \underline{j}}\sum_{{i<k}\atop{i\not\in \underline{j}}} \widehat\delta_{s(k),s(i)}.\]

It is important that a different ordered decomposition of the underlying representation can create a very different equivariant Schubert cell construction. For example, while $\Omega_{[2,5]}(\R^{+-+-+})\iso e^{4,2}$, the decomposition $\R^{5,2}=\R^{+-++-}$ gives $\Omega_{[2,5]}(\R^{+-++-})\iso e^{4,4}$, an ingredient for building the same space $\Gr_2(\R^{5,2})$ which does not appear in the $\R^{+-+-+}$ construction. 

A representation-cell structure for a space allows for a \textbf{one-cell-at-at-time filtration}, such that each subsequent inclusion cofiber is a representation sphere:

\[\xymatrix{
	X_0\ar[r] & X_1\ar[r]\ar[d] & X_2\ar[r]\ar[d] & \dots\ar[r]&X_i\ar[r]\ar[d]&\dots\\
    &S^{p_1,q_1}&S^{p_2,q_2}&\dots&S^{p_i,q_i}&\dots
}\]

This gives rise to the \textbf{one-cell-at-a-time equivariant cellular spectral sequence} for a Grassmannian, which we will discuss further in the next section. To a given choice of decomposition for the underlying representation space, we get a spectral sequence having for its $E_1$ page a free $\Mt$-module with basis elements corresponding to the bidegrees $(p_i,q_i)$ of these Schubert cells. We will refer to this data as a \textbf{table of ingredients} where each Young diagram or jump sequence represents the generator for an $\Mt$ in that bidegree. Denote the ingredient table of a certain decomposition $\bigoplus \R^{\pm}$ by $I(\pm\dots\pm)$. Figure \ref{fig:252ingreds pmpmp} has two depictions of $I(\R^{+-+-+})=I(+-+-+)$.

\begin{figure}[h]
	\scalebox{0.9}{\begin{tabular}{||l|p{10pt}|p{10pt}|p{15pt}|p{15pt}|p{15pt}|p{10pt}}
	\hspace{-.26in}3&&&\scalebox{.5}{\young(-,--)}     &&\scalebox{.5}{\young(+-,-+-)}     &\scalebox{.5}{\young(-+-,+-+)}\\
	\hline
	\hspace{-.26in}2&&&\scalebox{.5}{\young(-+-)}\,     &\scalebox{.5}{\young(-,++-)}&&\\
	&&&&      \scalebox{.5}{\young(+-,-+)}&&\\
	\hline
	\hspace{-.26in}1&\scalebox{.5}{\young(+)}      &\scalebox{.5}{\young(+-)}&&&&\\
	&&  \scalebox{.5}{\young(-,+)}&&&&\\
	\hline
	$\emptyset$\hspace{-.34in}0\hspace{.4in}\,&&&&&&\\
	\hline
	\hline
	\end{tabular}}
	\quad or \qquad
	\scalebox{0.7}[.7]{\begin{tabular}{||l|l|l|l|l|l|l}
		\hspace{-.26in}&&&[2,4]&&[3,5]&[4,5]\\
		\hline
		\hspace{-.26in}&&&[1,5]&[2,5]&&\\
		\hspace{-.26in}&&&&[3,4]&&\\
		\hline
		\hspace{-.26in}&[1,3]&[1,4]&&&&\\
		\hspace{-.26in}&&[2,3]&&&&\\
		\hline
		[1,2]\hspace{-.44in}\hspace{.4in}\,&&&&&&\\
		\hline
		\hline
	\end{tabular}}\\
	\scalebox{0.9}{\begin{tabular}{p{10pt}p{10pt}p{10pt}p{10pt}p{15pt}p{15pt}p{15pt}p{10pt}}
		&0&1&2&3&4&5&6
	\end{tabular}}
	\hspace{.9in}
	\begin{tabular}{p{11pt}p{11pt}p{11pt}p{11pt}p{11pt}p{11pt}}
		&&&&&
	\end{tabular}\\
	Young diagrams\qquad\qquad\qquad\qquad\qquad Jump sequences
	\caption{Ingredients table $I(+-+-+)$ for $\Gr_2(\R^{5,2})$.}
	\label{fig:252ingreds pmpmp}
\end{figure}

While the ingredients table is the first page of a spectral sequence, we will often make use of this data in another way. If we consider attaching these equivariant cells successively by increasing dimension and then weight, we can compute the cohomology of filtered subspaces one at a time. That is, rather than running a spectral sequence, we will repeatedly consider the long exact sequence corresponding to iteratively building subspaces $X_{k+1}$ from $X_k$ by attaching one equivariant cell $e^{p,q}$, giving the cofiber sequence
\[X_{k}\inj X_{k+1}\to \S pq.\] 

Because the differentials $d:H_{\text{sing}}^\bullet (X_k)\to H_{\text{sing}}^{\bullet+1}(S^p)$ in the non-equivariant chain complex are all zero, we know that none of the equivariant differentials may send a free generator to another free generator, as the forgetful map induces a natural map between the equivariant and non-equivariant long exact sequences for each attachment. Also because we are attaching cells by increasing weight, any differential carrying a generator into the top cone would hit $\tau^j$ times some other generator, which would again imply an isomorphism non-equivariantly.  Because the differentials in a Schubert cell complex for a Grassmannian must have zero differentials as their non-equivariant ``shadows,'' we need only worry about nonzero differentials into the lower cones of suspensions of $\Mt$, which, if they occur, cause Kronholm shifts. \par

We return to $\Gr_2(\R^{5,2})$, again recalling that rather than running a spectral sequence, we are simply computing the cohomology of subspaces as we attach cells one at at time. From Figure \ref{fig:252ingreds pmpmp} we can see that as we attach the first few cells, no differentials are possible, and so the cohomologies of early subspaces are obvious. But when the cell $\Omega_{\scalebox{0.4}{\yng(1,2)}}$ is attached, a differential between $\langle[\scalebox{.3}{\yng(1,1)}],[\scalebox{.3}{\yng(2)}]\rangle$ and $\langle\theta[\scalebox{.3}{\yng(1,2)}]\rangle$ could be either zero or nonzero without contradicting what is known non-equivariantly. However we can resolve this ambiguity by making use of another ordered decomposition of $\R^{5,2}$, the ambient space for our 2-planes. For example we have an equivariant homeomorphism $\Gr_2(\R^{+-+-+})\iso \Gr_2(\R^{-++-+})$, induced by the linear map $(x_1,x_2,x_3,x_4,x_5)\mapsto (x_2,x_1,x_3,x_4,x_5)$ on the underlying representation. This second construction for the space has ingredients table $I(\R^{-++-+})$, as shown in Figure \ref{fig:252ingreds mppmp}. Again, we can represent cells using either Young diagrams or jump sequences.
\begin{figure}
	\scalebox{0.9}{\begin{tabular}{||l|p{10pt}|p{10pt}|p{15pt}|p{15pt}|p{15pt}|p{10pt}}
		\hspace{-.26in}3&&&&\scalebox{0.5}{\young(-,-+-)}&\scalebox{0.5}{\young(-+,-+-)}&\scalebox{0.5}{\young(+--,-++)}\\
		\hline
		\hspace{-.26in}2&&\scalebox{0.5}{\young(--)}&\scalebox{0.5}{\young(-,+-)}&\scalebox{0.5}{\young(-+,+-)}&&\\
		 && \scalebox{0.5}{\young(-,-)}& & &&\\
		\hline
		\hspace{-.26in}1&&&\scalebox{0.5}{\young(++-)}&&&\\
		\hline
		$\emptyset$\hspace{-.33in}0\hspace{.29in}\,&\scalebox{0.5}{\young(+)}&&&&&\\
		\hline
		\hline
	\end{tabular}}
	\qquad
	\scalebox{0.7}[.7]{\begin{tabular}{||l|l|l|l|l|l|l}
		\hspace{-.26in}&&&&[2,5]&[3,5]&[4,5]\\
		\hline
		&&[1,4]&[2,4]&[3,4]&&\\
		\hspace{-.26in}&&[2,3]&&&&\\
		\hline
		\hspace{-.26in}&&&[1,5]&&&\\
		\hline
		[1,2]\hspace{-.44in}\hspace{.4in}\,&[1,3]&&&&&\\
		\hline
		\hline
	\end{tabular}}\\
	\,\,
	\scalebox{0.9}{\begin{tabular}{p{10pt}p{10pt}p{10pt}p{10pt}p{15pt}p{15pt}p{15pt}p{10pt}}
		&0&1&2&3&4&5&6
	\end{tabular}}
	\,\,\,\,
	\begin{tabular}{p{11pt}p{14pt}p{14pt}p{14pt}p{14pt}p{14pt}p{1pt}p{1pt}}
		&&&&&&&
	\end{tabular}
	\caption{Ingredients table $I(-++-+)$ for $\Gr_2(\R^{5,2})$.}
	\label{fig:252ingreds mppmp}
\end{figure}

Since after iteratively attaching these ingredients, we must arrive at the same cohomology, it is now clear that in this second scenario, $[1,3]$ must ``shift up'' by hitting some nonzero combination of $\theta[1,4]$ and $\theta[2,3]$, after which no other differential can interact with the bidegree $(2,1)$, recalling that isomorphisms are precluded by our knowledge of the non-equivariant cochain complex. Thus in the first construction, $d:\langle [1,4],[2,3]\rangle \to \langle\theta[2,4]\rangle$ must be nonzero, so that both $\H21$ and $\H22$ of $\Gr_2(\R^{+-+-+})$ contain generators. As no other differentials are possible in the $+-+-+$ construction, we now know that
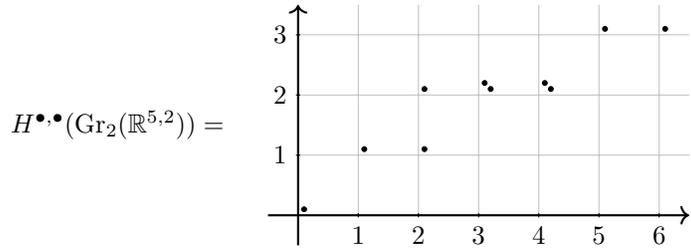
\begin{center}
	\begin{tikzpicture}[scale=.8]
		\def\pmax{ 6 }
		\def\qmax{ 3 }
		\draw (-3,1.5) node {$H^{\bullet,\bullet}(\Gr_2(\R^{5,2}))=$};
		\def\points{(0.1, 0.1), (1.1, 1.1), (2.1, 1.1), (2.1, 2.1), (3.2, 2.1), (3.1, 2.2), (4.2, 2.1), (4.1, 2.2), (5.1, 3.1), (6.1, 3.1)}
		\draw[step=1cm,lightgray,very thin] (0,0) grid (\pmax+0.5,\qmax+.5);
		\draw[thick,->] (0,-0.5) -- (0,.5+\qmax);
		\draw[thick,->] (-0.5,0) -- (.5+\pmax,0);
		\foreach \x in {1,...,\pmax}
			\draw (\x cm,1pt) -- (\x cm,-1pt) node[anchor=north] {$\x$};
		\foreach \y in {1,...,\qmax}
			\draw (1pt,\y cm) -- (-1pt,\y cm) node[anchor=east] {$\y$};
		\foreach \pq in \points
			\draw \pq [fill] circle (.04);
	\end{tikzpicture}.\label{example}
\end{center}
(It is instructive to check what this must mean about the other differentials in the $-++-+$ construction.\footnote{We must have $[1,3]\mapsto\theta[1,4]$ and $[1,5]\mapsto\theta[2,5]$.})

This procedure of playing the many different constructions for a Grassmannian off of one another can be automated to get a fund of examples. The theorems and algorithm necessary for this will be described in a forthcoming paper. But in many cases, we can do even better -- see Sections \ref{knone} and \ref{twontwo}.

\section{Jack-O-Lantern Modules}
\label{jolsection}
Rather than considering successive long exact sequences as in Section \ref{we1}, we could have used a cellular spectral sequence made by sewing together the long exact sequences for each cofiber sequence in the filtration
	\[\xymatrix{
	\pt\ar[r]& X_1\ar[r]\ar[d]& X_2\ar[d]\\
	 &{\color{blue}\S10}&{\color{ggreen}\S22}
	}.\]
More generally, when a space $X$ is built one-cell-at-a-time, so that the cofiber of each subspace inclusion is a single representation sphere,
\[\xymatrix{
\pt\ar[r]&X_1\ar[r]\ar[d]&X_2\ar[r]\ar[d]&X_3\ar[r]\ar[d]&\dots\\
&\S{p_1}{q_1}&\S{p_2}{q_2}&\S{p_3}{q_3}&\dots
}\]
we can make a spectral sequence where each filtration degree contains a single suspended $\Mt$. This spectral sequence is, alarmingly, trigraded, but if we attach cells in lexicographic order, we can suppress the filtration degree without losing too much information. Letting $r$ denote filtration degree, we will have differentials $$d_k:E^{p,q,r}_k\to E^{p+1,q,r+k}_k.$$ Figure \ref{fig:spectraljack} depicts this approach for one of the constructions in Section \ref{we1}.

\begin{figure}[h]
	\begin{center}
		\begin{tikzpicture}[scale=0.55]
				\axisname{$E_1$}
				\draw[->] (1.2,0)-- node[above] {$d_1$}(1.8,0);
				\color{blue}
				\HMtwo{1}{0}
				\color{ggreen}
				\draw[->] (2,2.2)--(2,4.5);
				\draw[->] (2,2.2)--(2+2.3,4.5);
				\draw[->] (2,0)--(2,-4.5);
				\draw[->] (2,0)--(2-4.5,-4.5);
		\end{tikzpicture}
		\begin{tikzpicture}[scale=0.55]
				\axisname{$E_2=E_\infty$}
				\color{blue}
				\draw[->] (1,1)--(1,4.5);
				\draw[-] (1,1)--(2,2);
				\draw[-] (2,2)--(2,1);
				\draw[->] (2,1)--(2+2.5,1+2.5);
				\draw[->] (1,-2.1)--(1,-4.5);
				\draw[->] (1,-2.1)--(1-2.4,-4.5);
				\color{ggreen}
				\draw[->] (2,2.1)--(2,4.5);
				\draw[->] (2,2.1)--(2.1+2.3,4.5);
				\draw[->] (1,-1)--(1-3.5,-4.5);
				\draw[-] (1,-1)--(1,-2);
				\draw[-] (1,-2)--(2,-1);
				\draw[->] (2,-1)--(2,-4.5);
		\end{tikzpicture}
	\end{center}
	\caption{Jack-O-lantern modules in a spectral sequence.}
	\label{fig:spectraljack}
\end{figure}

While we already know that the reduced cohomology of the space $\Gr_1(\R^{3,1})$ from Section \ref{we1} is the free module $\M11\oplus \M21$, we see that $E_\infty$ is not itself free. Rather, it is an associated graded of this free module. Loosely, the summands of $E_\infty$ are copies of $\Mt$ with pieces cut out of them. This phenomenon motivates the following definition.

\begin{defn}
	\label{joldef}
	Beginning with some suspension $\Sigma^{p,q}\Mt=\Mt\langle a\rangle$ of $\Mt$, let $S$ be a finite set of homogeneous elements of the lower cone, and consider the quotient $\sfrac{\Mt\langle a \rangle}{S\Mt\langle a \rangle}$. Let $J$ be a submodule of this quotient generated by a finite collection of homogeneous elements of the upper cone, and let these generators include elements of the form $[\rho^M a]$ and $[\tau^Na]$ for some $M$ and $N$. A \textbf{Jack-o-lantern} module is an $\Mt$ module isomorphic to such a module $J$. (See Figure \ref{fig:orange} for an example.)\par
	We can decompose a \jol module as $J=J^+\sqcup J^-$ where the module structure connects $J^+$ to $J^-$. These two parts are
	\begin{itemize}
		\item An ideal $J^+$ of the upper cone of $\Mt\langle a\rangle$ such that for large enough $N$ and $M$, both $[\rho^Ma]\in J^+$ and $[\tau^Na]\in J^+$\\
		and
		\item A ``coideal''  $J^-$ of the lower cone of $\Mt\langle a\rangle$, meaning that if $[\frac\theta{\rho^i\tau^j}a]\in J^-$ then both $[\frac\theta{\rho^{i+1}\tau^j}a]\in J^-$ and $[\frac\theta{\rho^i\tau^{j+1}}]a\in J^-$, such that for large enough $N$ and $M$, both $[\frac\theta{\rho^M}a]\in J^-$ and $[\frac\theta{\tau^N}]a\in J^-$.
	\end{itemize}
	 By ``the module structure connects $J^+$ to $J^-$'' we mean if $[\rho^k\tau^la],[\frac\theta{\rho^i\tau^j}a]\in J$, then $\frac\theta{\rho^{i+k}\tau^{j+l}}\cdot[\rho^k\tau^la]=[\frac\theta{\rho^i\tau^j}a]$.
\end{defn}

Note that $\Mt$ itself is trivially a \jol module. These modules contain elements of the form $[\tau^N a]$ and $[\frac\theta{\tau^N}a]$ sharing a dimension $p$. Likewise there is the largest fixed-set dimension $p-q$ of $J^+$. Together these give $J$ a well-defined \textbf{phantom bidegree}, the pair $(p,q)$ such that $J$ is a subquotient of $\Sigma^{p,q}\Mt$. See Figure \ref{fig:orange}. Note that while we may write elements like $[\tau\rho a]$, in general there my be no element $[a]$, that is, often $J^{p,q}=0$. We will sometimes call $a$ the \textbf{phantom generator}.

\begin{figure}[h]
	\begin{tikzpicture}[scale=0.3]
			\color{gray}
			\draw[dashed] (0,0) -- (0,4);
			\draw[dashed] (0,0) -- (5,5);
			\draw[dashed] (0,-2) -- (-5,-7);
			\draw[dashed] (0,-2) -- (0,-8);
			\color{black}
			\draw[thick,->] (0,4) -- (0,7);
			\draw[thick,->] (5,5) -- (7,7);
			\draw[thick, fill=orange] (0,6.5)--(0,4) -- (1,5)--(1,4)--(2,5)--(2,3)--(5,6)--(5,5)--(6.5,6.5);
			\draw[thick,->] (-5,-7) -- (-9,-11);
			\draw[thick, fill=orange] (-8.5,-10.5)--(-5,-7) -- (-5,-9)-- (-3,-7)-- (-3,-10)-- (-1,-8)-- (-1,-9) -- (0,-8)-- (-0,-10.5);
			\draw[thick,->] (0,-8) -- (-0,-11);
			\draw[->] (-2,0)--(-.2,0);
			\draw (-2,0) node[left] {phantom bidegree};
			\draw (5,5) node {$\bullet$};
			\draw (5,5) node[below right] {$[\rho^5 a]$};
			\draw (0,4) node {$\bullet$};
			\draw (0,4) node[left] {$[\tau^4 a]$};
			\draw (-5,-7) node {$\bullet$};
			\draw (-5,-7) node[above left] {$[\frac\theta{\rho^5} a]$};
			\draw (0,-8) node {$\bullet$};
			\draw (-0,-8) node[right] {$[\frac\theta{\tau^6} a]$};
		\end{tikzpicture}
		\caption{A \jol module. In the language of definition \ref{joldef}, this is $\left(\sfrac{\Mt\langle a \rangle}{S\Mt\langle a \rangle}\right)\langle [\tau^4a],[\rho\tau^3a],[\rho^2\tau a],[\rho^5a]\rangle$ where $S=\{\frac\theta{\rho^4\tau}a,\frac\theta{\rho^2\tau^4}a,\frac\theta{\tau^5}a\}$.}
		\label{fig:orange}
\end{figure}

\begin{defn}
	\label{jolmap}
	A \textbf{\jol map} is an $\Mt$-module map $f:J_1\to J_2$ between \jol modules such that $f(J_1^+)\subseteq J_2^-$ and hence $f(J_1^-)=0$. \par
	Note that the kernel and cokernel of a \jol map are both \jol modules, with $ J_1^+\surj (\ker f)^+$ and $(\ker f)^-\iso J_1^-$ while $(\cok f)^+\iso J_2^+$ and $ J_2^-\surj (\cok f)^-$. See for example Figure \ref{fig:jolkercok}.
\end{defn}

\begin{figure}
	\begin{tikzpicture}[scale=0.5]
		\draw [->] (1.2,1)--node[above right]{$f$}(1.8,1);
		\color{ggreen}
		\draw (0,0) node {$J_1$};
		\draw [<->] (0,4)--(0,1)--(1,2)--(1,1)--(4,4);
		\draw [<->] (-3,-5)--(-1,-3)--(-1,-4)--(0,-3)--(0,-5);
		\color{orange}
		\draw (3,4) node {$J_2$};
		\draw [<->] (3,7)--(3,5)--(4,6)--(4,5)--(6,7);
		\draw [<->] (0,-1)--(2,1)--(2,0)--(3,1)--(3,-1);
	\end{tikzpicture}
	\quad
	\begin{tikzpicture}[scale=0.5]
		\draw [dashed] (0,7)--(0,-5);
	\end{tikzpicture}
	\quad
	\begin{tikzpicture}[scale=0.5]
		\color{ggreen}
		\draw (0,0) node[left] {$\ker(f)$};
		\draw [<->] (0,4)--(0,1)--(2,3)--(2,2)--(4,4);
		\draw [<->] (-3,-5)--(-1,-3)--(-1,-4)--(0,-3)--(0,-5);
		\color{orange}
		\draw (3,4) node[left] {$\cok(f)$};
		\draw [<->] (3,7)--(3,5)--(4,6)--(4,5)--(6,7);
		\draw [<->] (0,-1)--(1,0)--(1,-1)--(3,1)--(3,-1);
	\end{tikzpicture}
	\caption{A \jol map $f$ of degree $(1,0)$. The kernel and cokernel of $f$ are also \jol modules.}
	\label{fig:jolkercok}
\end{figure}

\begin{lem}
	\label{jackslemma}
	Let $C_\bullet=\left[\dots\to J_{i-1}\xrightarrow{d_{i-1}}J_i \xrightarrow{d_i} J_{i+1}\xrightarrow{d_{i+1}} \dots \right]$ be a (co)chain complex of \jol maps. $($Definition \ref{jolmap} implies $d^2=0)$. Then the homology modules $H_i(C_\bullet)$ are \jol modules.
\end{lem}

\begin{proof}
	The kernel and cokernel of a \jol map are both \jol modules. As $J_{i-1}\to \ker(d_i)$ is a \jol map, we have that $H_i(C_\bullet)=\cok(J_{i-1}\to \ker(d_i))$ is a \jol module.
\end{proof}

In the next section we will be interested in maps between \jol modules which are $\Mt$-module maps but not necessarily \jol maps. The following lemma gives a restriction on the topological and fixed-set dimensions of the phantom generators of the domain and codomain of such maps $f:J\to J'$ with $f(J^+) \subseteq (J')^+$.

\begin{lem}
	\label{toptotop}
	Suppose $J$ and $J'$ are \jol modules having phantom generators $\alpha$ and $\alpha'$ of phantom bidegrees $(p,q)$ and $(p',q')$ respectively. If $f:J\to J'$ is a degree-preserving $\Mt$-module map carrying an element of $J^+$ to a nonzero element of $(J')^+$, then $f|_{J^+}$ is injective, $p\ge p'$ and $p-q\le p'-q'$. (Note this implies $q\ge q'$).
\end{lem}

\begin{proof}
	Say $f([\rho^i\tau^j \alpha])= [\rho^{i'}\tau^{j'} {\alpha'}]$ but suppose there exists $[\rho^k\tau^l \alpha]\in J^+$ with $f([\rho^k\tau^l\alpha])=0$. Then
	
	\[\rho^k\tau^l[\rho^{i'}\tau^{j'} {\alpha'}]=\rho^k\tau^lf([\rho^i\tau^j \alpha])=\rho^i\tau^jf([\rho^k\tau^l \alpha])=\rho^i\tau^j\cdot 0=0.\]
	
	This makes $[\rho^{i'}\tau^{j'}\alpha']$ an element of $(J')^+$ with $\Mt^+$ torsion, a contradiction.
	
	Now suppose $p<p'$. There exists some $M$ such that $[\tau^M\alpha]\in J$. This element lies to the left of $(J')^+$ since every element here has topological degree at least $p'$. Thus $f([\tau^M\alpha])=0$, contradicting injectivity.
	
	Similarly, suppose $p-q> p'-q'$ and $[\rho^M\alpha]\in J$. Noting that $[\rho^M\alpha]$ lies on a diagonal below $(J')^+$, we have $f([\rho^M\alpha])=0$, a contradiction.
\end{proof}


In the next section and throughout the paper, we will be examining differentials which are $\Mt$-module maps of degree $(1,0)$, and we will be interested in maps between \jol modules of certain degrees. This motivates the following corollary.

\begin{cor}
	\label{fixedcor}
	Suppose $J$ and $J'$ are \jol modules having phantom generators $\alpha$ and $\alpha'$ of phantom bidegrees $(p,q)$ and $(p',q')$ respectively, with $(p,q)\le (p',q')$ lexicographically. If $d$ is an $\Mt$-module map of degree $(1,0)$ sending an element of $J^+$ to a nonzero element of $(J')^+$, then $p=p'-1$ and for large enough $N$, $d([\tau^N \alpha])=[\tau^{N+q-q'}\alpha']$. 
\end{cor}

\begin{proof}
	Write $d$ as a degree-preserving $\Mt$-module map $d:\Sigma^{1,0}J\to J'$. Note that $p+1\ge p'$ by Lemma \ref{toptotop}, but $p\le p'$ by the lexicographic ordering. So either $p=p'$ or $p=p'-1$. 
	
	Suppose for a contradiction that $p=p'$. Then $q\le q'$ by the lexicographic ordering. But by the lemma, $(p+1)-q\le p'-q'$. So $q\ge q'+1$ and we have our contradiction. Thus $p=p'-1$.
	
	Finally for some $N$ we have $[\tau^N\alpha]\ne 0$, and $d([\tau^N\alpha])=[\tau^{N+q-q'}\alpha']$ by the injectivity given by the lemma.
\end{proof}

If a map from $\Mt$ is zero on the generator, it is identically zero, and in particular, zero on the lower cone. For \jol modules, we have the analogous fact: 

\begin{lem}\label{notopnobot}
	If $f:J\to K$ is an $\Mt$-module map from a \jol module $J$ to any $\Mt$ module $K$, and $f|_{J^+}$ is not injective, then $f(J^-)=0$.
\end{lem}

\begin{proof}
	By assumption, for some $M$ and $N$ we have $[\rho^M\tau^N1_\alpha]\in J^+$ with $f([\rho^M\tau^N1_\alpha])=0$. Then for any element $[\frac\theta{\rho^i\tau^j}1_\alpha]\in J^-$,
	$$f\left(\left[\tfrac\theta{\rho^i\tau^j}1_\alpha\right]\right)
	= f\left(\tfrac\theta{\rho^{M+i}\tau^{N+j}}\cdot [\rho^M\tau^N1_\alpha]\right)
	=\tfrac\theta{\rho^{M+i}\tau^{N+j}} f\left([\rho^M\tau^N1_\alpha]\right)=0.$$
\end{proof}

In the next section, we put these algebraic results to use in cohomology.

\section{A theorem restricting Kronholm shifts}

Because $\Mt$ has nonzero groups in so many bidegrees, when examining spectral sequences or even just long exact sequences of a pair, there are in general a lot of algebraically-possible differentials to consider. The following theorem helps us rule out some of these possibilities. We first need a definition.

\begin{defn}
	Let $\Lambda$ be a set with a partial ordering so that any two elements $\alpha,\beta\in \Lambda$ have a greatest lower bound, denoted $\alpha\cap\beta\in \Lambda$. A \textbf{hierarchical cell structure} on a space $X$ is a CW structure with cells $\{e_\lambda\}_{\lambda\in\Lambda}$ so that each cell $e_\alpha$ has an attaching map whose image lies in $\ds\bigsqcup_{\lambda<\alpha}e_\lambda$. 
\end{defn}
Note that a hierarchical cell structure on a space gives subspaces $X_\alpha$ of the form
$$X_\alpha=\bigsqcup_{\lambda\le \alpha}e_\lambda$$
with the properties
\begin{itemize}
		\item $X_\alpha\cap X_{\beta}=X_{\alpha\cap\beta}$\quad and
        \item $\displaystyle\sfrac{(X_\alpha\cup X_{\alpha'})}{X_{\alpha}\cap X_{\alpha'}}=(\sfrac{X_\alpha}{(X_{\alpha}\cap X_{\alpha'})})\vee (\sfrac{X_{\alpha'}}{(X_{\alpha}\cap X_{\alpha'})}).$
\end{itemize}

\subsection{Example}
	The setting in which we will use this notion is of course the Schubert cell construction of the Grassmannian, which gives it a hierarchical cell structure, by Proposition 3.2.3 of \cite{manivel}. Schubert cells are indexed by Young diagrams, which have a partial ordering under inclusion. In fact, in this setting,
    \[X_\lambda=\bigsqcup_{\lambda'\le\lambda}\Omega_{\lambda'}=\overline{\Omega_\lambda},\]
    the closure of the Schubert cell $\Omega_\lambda$, called the \textbf{Schubert variety}.
    For two diagrams $\lambda_1$ and $\lambda_2$, 
    $$X_{\lambda_1}\cap X_{\lambda_1}=X_{\lambda_1\cap\lambda_2}$$ 
	and 
	$$\sfrac{(X_{\lambda_1}\cup X_{\lambda_2})}{X_{\lambda_1\cap\lambda_2}}=(\sfrac{X_{\lambda_1}}{X_{\lambda_1\cap\lambda_2}})\vee (\sfrac{X_{\lambda_2}}{X_{\lambda_1\cap\lambda_2}})$$
	for example in $\Gr_2\R^5$ (see Figure \ref{fig:threeversions})
	$$X_{\scalebox{0.4}{\yng(1,3)}}\cap X_{\scalebox{0.4}{\yng(2,2)}}=X_{\scalebox{0.3}{\yng(1,3)}\,\cap\,\scalebox{0.3}{\yng(2,2)}}=X_{\scalebox{0.3}{\yng(1,2)}}$$ 
		and 
		$$\sfrac{(X_{\scalebox{0.4}{\yng(1,3)}}\cup X_{\scalebox{0.4}{\yng(2,2)}})}{X_{\scalebox{0.4}{\yng(1,2)}}}
		=(\sfrac{X_{\scalebox{0.4}{\yng(1,3)}}}{X_{\scalebox{0.4}{\yng(1,2)}}})\vee (\sfrac{X_{\scalebox{0.4}{\yng(2,2)}}}{X_{\scalebox{0.4}{\yng(1,2)}}}).$$\\

\subsection{Warning}
To avoid confusion, we call the reader's attention to the fact that there are two distinct orderings in the following theorem -- a partial ordering corresponding to a hierarchical cell structure, and the lexicographic order on bidegrees. 

We now state the theorem we will use in Section \ref{knone} to rule out Kronholm shifts.

\begin{thm}
\label{fixedthm}
	Let $X'$ be an equivariant space with a finite hierarchical representation cell structure. Suppose that $X$ is obtained from a subcomplex $X'\subset X$ by attaching a single representation cell $e_\beta\simeq e^{p,q}$ whose bidegree is lexicographically after that of the cells of $X'$.
Suppose also that the forgetful cochain complex $C^\bullet(\psi(X))$ corresponding to this construction has only zero differentials. If, for every cell $e_\alpha\simeq e^{p',q'}$ used in building $X'$, either
		\begin{enumerate}
			\item[$(i)$] $\alpha\not\le \beta$ \qquad or
			\item[$(ii)$] $p'-q'\le p-q$
		\end{enumerate}
	then the cofiber sequence $X'\inj X\to \S pq$ gives a split short exact sequence in cohomology, i.e.
	\[\H \bullet\bullet (X)=\H\bullet\bullet (X')\oplus \M pq.\]
	$\ $\\
\end{thm}

To prove this theorem, we need three lemmas. The first makes use of lexicographic ordering to establish that the spectral sequence of an equivariant space whose attaching maps are forgetfully trivial will be made up of \jol modules.

	\begin{lem}\label{jackolemma}
		Suppose an equivariant representation-cell complex $X$ is built by attaching cells in lexicographic order. Suppose the associated forgetful chain complex $C^\bullet(\psi(X))$ has all zero differentials. Then the corresponding one-cell-at-a-time equivariant spectral sequence $E_\bullet^{\bullet,\bullet,\bullet}$ with the  lexicographic filtration of $X$ has a \jol module for the $r^{\text{th}}$ filtration of its $k^{\text{th}}$ page, $E_k^{\bullet,\bullet,r}$ for all $k$ and $r$. Furthermore, each differential $d_k:E^{\bullet,\bullet,r}_k\to E^{\bullet+1,\bullet,r+k}_k$ is a \jol map.
	\end{lem}

While the previous lemma admits the possibility of nonzero differentials one might not expect from looking at the forgetful data, the next lemma makes use of the hierarchical structure to establish a condition in which we need not worry about surprise differentials.  

	\begin{lem}\label{alphabeta}
		Under the assumptions of the theorem, if cells $e_\alpha$ and then $e_\beta$ are used in building a space $X$, but $\alpha\not\le\beta$, then there is no differential from the filtration degree of $\alpha$ to that of $\beta$ in the one-cell-at-a-time spectral sequence for $X$.
	\end{lem}

Our third lemma for this theorem establishes a fact we will use to split a long exact sequence into short exact sequences in the proof of the theorem.

	\begin{lem}\label{seqs}
		Let $X$ be a finite filtered space $$\pt=X_0\subseteq X_1\subseteq\dots\subseteq X_{n-1}\subseteq X_n=X.$$ If every differential of the associated spectral sequence mapping into the $\sfrac {X_{n}}{X_{n-1}}$ filtration is zero, then for the cofiber sequence $X_{n-1}\xrightarrow{i}X_n\to \sfrac{X_n}{X_{n-1}}$, the map $i^*:\H\bullet\bullet (X_n)\to \H\bullet\bullet(X_{n-1})$ is surjective.
	\end{lem}

		We first prove the lemmas, and then the theorem. While we will not need the theorem until Section \ref{twontwo}, Lemmas \ref{jackolemma} and \ref{alphabeta} will be used in Section \ref{knone}.

		\begin{proof}[Proof of Lemma \ref{jackolemma}]
			Denote the generator of the $\Sigma^\alpha \Mt$ by $1_\alpha$. On page one of the spectral sequence there are upper cone elements of the form $\rho^i\tau^j 1_\alpha$ and lower cone elements $\frac\theta{\rho^i\tau^j}1_\alpha$. Denote\footnote{This is possible because the one-cell-at-a-time filtration precludes any elements of the form $[\rho^i\tau^j 1_\alpha+\frac\theta{\rho^{i'}\tau^{j'}}1_{\alpha'}]$, as these are filtration-inhomogeneous.} such elements surviving to later pages by $[\rho^i\tau^j 1_\alpha]$ and $[\frac\theta{\rho^i\tau^j}1_\alpha]$.\par
			Proceed by induction on the page of the spectral sequence. To begin with, $E_1=E_1^{\bullet,\bullet,\bullet}$ consists of a single suspension of $\Mt$ in each nonzero filtration. Differentials are determined by the image of each $1_{\alpha}$. Because of our lexicographic filtration, the only possible top-to-top differential is of the form $d_1(1_{\alpha})=\tau^j 1_{\alpha'}$. However, as $\psi(\tau)=1$, this would mean a nonzero differential in the forgetful setting, a contradiction of our assumption. Thus the differential $d_1$ forms a complex of (trivial) \jol modules. Hence by Lemma \ref{jackslemma}, $E_2$ consists of a \jol module in each filtration. See for example Figure \ref{fig:zig}.

	\begin{figure}[h]
		\begin{tikzpicture}[scale=0.4]
			\color{cyan}
			\draw[->] (2.2,1)--(0.2,-1);
			\draw[->] (2.2,1)--(2.2,-1);
			\draw[->] (2.2,3)--(2.2,4);
			\draw[->] (2.2,3)--(3.2,4);
			\color{green}
			\draw[->] (-2,-5)--(-2,-3.5);
			\draw[->] (-2,-5)--(-0.5,-3.5);
			\draw[->] (-2,-7)--(-2,-8);
			\draw[->] (-2,-7)--(-3,-8);
			\color{black}
			\draw[thick, ->] (0,-2) -- (0,-8);
			\draw[thick, ->] (0,-2) -- (-6,-8);
			\draw[thick,->] (0,0) -- (4,4);
			\draw[thick,->] (0,0) -- (0,4);
			\draw (0,-9.5) node {$E_1$};
			\draw (6,-1) node {$\rightsquigarrow$};
			\draw (.6,-.1) node {$\to$};
			\draw (-2+.6,-5) node {$\to$};
		\end{tikzpicture}
		\quad
		\begin{tikzpicture}[scale=0.4]
			\color{gray}
			\draw[dashed] (0,0) -- (2,2);
			\draw[dashed] (0,0) -- (0,1);
			\draw[dashed] (0,-2) -- (0,-4);
			\draw[dashed] (0,-2) -- (-2,-4);
			\color{black}
			\draw[thick,->] (2,2) -- (4,4);
			\draw[thick,->] (0,1) -- (0,4);
			\draw[thick] (0,1) -- (2,3);
			\draw[thick] (2,2) -- (2,3);
			\draw[thick, ->] (0,-4) -- (0,-8);
			\draw[thick] (0,-4) -- (-2,-6);
			\draw[thick] (-2,-4) -- (-2,-6);
			\draw[thick, ->] (-2,-4) -- (-6,-8);
			\fill[fill=black] (2,3) circle (2mm) node[above] {$[\rho^2\tau 1_\alpha]$};
			\fill[fill=black] (-2,-4) circle (2mm) node[above left] {$[\frac\theta{\rho^2} 1_\alpha]$};
			\draw (0,-9.5) node {$E_2$};
		\end{tikzpicture}
		\caption{A filtration degree of $E_1$ with differential in and out of that degree, and the corresponding filtration on the $E_2$ page. Note a connection between upper an lower cones remains. While $[1_\alpha]=0$, it is still the case, for example, that $\frac\theta{\rho^4\tau}\cdot [\rho^2\tau 1_\alpha]=[\frac\theta{\rho^2} 1_\alpha]$. This filtration degree of $E_2^{\bullet,\bullet,r}$ is a \jol module.}
		\label{fig:zig}
	\end{figure}
			
		Now assume for induction that the page $E_k=E_k^{\bullet,\bullet,\bullet}$ consists of a \jol module in each nonzero filtration. We must show that the differentials $d_k$ are \jol maps. First, there can be no top-to-top map on the $E_k$ page: By Corollary \ref{fixedcor}, such a map would carry a nonzero element $[\tau^{N} 1_\alpha]$ to the nonzero element $[\tau^{N'} 1_{\alpha'}]$, where $N'=N+(w(\alpha)-w(\alpha'))$. But considering the forgetful map $\psi$ yields a contradiction, as we see a nonzero differential $\psi(d_k):1_\alpha\mapsto 1_{\alpha'}$, in the non-equivariant spectral sequence, contradicting our assumption about the forgetful chain complex.
		
		Now either $d_k$ is zero on the top cone, and hence identically zero by Lemma \ref{notopnobot} and hence trivially a \jol map, or else $d_k$ maps the top cone to the bottom cone, and is a \jol map since $\theta^2=0$. Again applying Lemma \ref{jackslemma}, $E_{k+1}$ consists of \jol modules. This completes the induction.\\	
		\end{proof}

		\begin{proof}[Proof of Lemma \ref{alphabeta}]
			For our space $X$ with a hierarchical cell structure satisfying the assumptions of the theorem, let $\alpha$ be the  index of a cell attached before the cell $e_\beta$ such that $\alpha\not\le \beta$, and define 
	        $$Y=X_\alpha\cup X_\beta=\left(\bigsqcup_{\lambda\le \alpha}e_\lambda\right)\cup \left(\bigsqcup_{\lambda\le \beta}e_\lambda\right)$$ 
	        and 
	        $$Z=X_\alpha\cap X_\beta=\bigsqcup_{\srac{\lambda<\alpha\text{ and }}{\lambda<\beta}}e_\lambda.$$
	    Then $\sfrac YZ$ is a wedge of spaces 
		$$\sfrac YZ=\sfrac {(X_\alpha)}{Z}\vee\sfrac{(X_\beta)}Z=:A\vee B$$ 
	and there are quotient maps $Y\to  A\vee B\to A$ which respect the filtration grading of the spectral sequence. Thus any element $[x]_Y\in E_k^{\bullet,\bullet,\bullet}(Y)$ with the filtration degree of $\alpha$ corresponds to an element $[x]_A\in E_k^{\bullet,\bullet,\bullet}(A)$. Because $\alpha$ is in the highest nontrivial filtration of $A$, $[x]_A$ is a permanent cycle. We now have the commuting diagram
		\begin{figure}[H]
			\[\xymatrix{
				 [x]_A\ar@{|->}[d]& E_k^{\bullet,\bullet,\bullet}(A)\ar[r]\ar[d]_{d_k} & E_k^{\bullet,\bullet,\bullet}(Y)\ar[d]_{d_k}&  [x]_Y \ar@{|->}[d]\\
						 0 & E_k^{\bullet+1,\bullet,\bullet+k}(A)\ar[r] & E_k^{\bullet+1,\bullet,\bullet+k}(Y)& 0
					}\]
		\end{figure}

	$\!\!\!\!\!\!\!\!\!$ and so in particular the differential between the filtration degrees of $\alpha$ and $\beta$ is zero in $E^{\bullet,\bullet,\bullet}_\bullet(Y)$. Now the inclusion $i:Y\inj X$ induces a surjective spectral sequence map $i^*:E^{\bullet,\bullet,\bullet}_1(X)\to E^{\bullet,\bullet,\bullet}_1(Y)$, and so a nonzero $\alpha$-to$-\beta$ differential in $E_\bullet^{\bullet,\bullet,\bullet}(X)$ would imply one in $E_\bullet^{\bullet,\bullet,\bullet}(Y)$. Thus no such differential can exist.
	\end{proof}

	Finally, we prove Lemma \ref{seqs} by a diagram chase:

	\begin{proof}[Proof of Lemma \ref{seqs}]
		Naming the inclusions $$\pt\xrightarrow{i_0}X_1\xrightarrow{i_1}\dots\xrightarrow{i_{n-2}}X_{n-1}\xrightarrow{i_{n-1}}X_n$$ recall that we build the spectral sequence by weaving together the long exact sequences of the cofiber sequences $X_{k}\xrightarrow{i_k} X_{k+1}\xrightarrow{q_k}\sfrac{X_k}{X_{k-1}}$.\\
		\begin{tabular}{c||c}
		\xymatrix{
			\dots H(\frac{X_{n}}{X_{n-1}})\ar[r]^{q_n^*}&HX_n\ar[d]^{i_{n-1}^*}\\
			\dots H(\frac{X_{n-1}}{X_{n-2}})\ar[r]^{q_{n-1}^*}&
			H(X_{n-1})\ar[d]^{i_{n-2}^*}\ar[r]^\delta&H(\frac{X_n}{X_{n-1}})\dots\\
			\dots H(\frac{X_{n-2}}{X_{n-3}})\ar[r]^{q_{n-2}^*}&
			H(X_{n-2})\ar[d]^{i_{n-3}^*}\ar[r]^\delta&H(\frac{X_{n-1}}{X_{n-2}})\dots\\
			\vdots&\vdots\ar[d]^{i_0^*}&\vdots\\
			\dots H(\frac{X_0}{X_{-1}})\ar[r]^{q_0^*}&HX_0\ar[r]^\delta\ar[d]^{i_{-1}^*}&H(\frac{X_1}{X_0})\dots\\
			&HX_{-1}
		}&
			\xymatrix{
				\ar[r]^{q_n^*}&HX_n\ar[d]^{i_{n-1}^*}\\
				b\ar@{|->}[r]^{q_{n-1}^*}&
				\,\phantom{\int}a\phantom{\int}\ar@{|->}[d]^{i_{n-2}^*}\ar[r]^\delta &H(\frac{X_n}{X_{n-1}})\\
				c\ar@{|->}[r]^{q_{n-2}^*}&
				\,\phantom{\int}i^*_{n-2}(a)\phantom{\int}\ar@{|->}[d]^{i_{n-3}^*}\ar[r]^\delta&\dots\\
				\,\phantom{\int}&\vdots\ar[d]^{i_0^*}&\vdots\\
				\ar[r]^{q_0^*}&\,HX_0\phantom{\int}\ar[r]^\delta\ar[d]^{i_{-1}^*}&\dots\\
				&0
		}
		\end{tabular}

	Assuming all differentials to $H\left(\frac{X_n}{X_{n-1}}\right)$ are zero, we wish to show that $i^*_{n-1}$ is surjective. Consider an element $a\in H(X_{n-1})$.\par 
	If $i_{n-2}^*(a)=0$, by exactness there exists $b$ with $q_{n-1}^*(b)=a$, and then by assumption $d_1(b)=(\delta\circ q_{n-1}^*)(b)=\delta(a)=0$ and so by exactness $a$ lies in the image of $i_{n-1}^*$.\par
	If $i_{n-2}^*(a)\ne 0$ but $i_{n-3}^*(i_{n-2}^*(a))=0$, then by exactness there is some $c$ so that $q_{n-2}^*(c)=i_{n-2}^*(a)$. Since $d_2(c)=\delta(a)=0$, we again have $a$ in the image of $i_{n-1}^*$ by exactness.\par
		Since $H(X_{-1})=0$, eventually we are guaranteed some $(i_k^*\circ i_{k+1}^*\circ\dots\circ i_{n-2}^*)(a)=0$ and hence some $x\in H(\frac{X_k}{X_{k-1}})$ such that $q_{k+1}^*(x)=( i_{k+1}^*\circ\dots\circ i_{n-2}^*)(a)$, and since $d_k(x)=\delta(a)=0$, again $a$ is mapped to by $i_{n-1}^*$. Hence $i_{n-1}^*$ is surjective.
	\end{proof}

We can now prove the theorem. 

	\begin{proof}[Proof of Theorem \ref{fixedthm}]
	Again filter $X$ one-cell-at-a-time lexicographically, meaning by increasing topological dimension and then increasing weight. Let $n$ denote the filtration degree of the cell added to $X'$ to form $X$. This filtration gives a trigraded spectral sequence, as discussed in Section \ref{jolsection}. We claim there are no differentials hitting any nonzero elements of the $n$th filtration degree $\Hbb(\sfrac{X_n}{X_{n-1}})=\Hbb(\sfrac{X}{X'})$. To see this, consider each lower filtration degree $k<n$, which corresponds to some $e_\alpha$ used in building $X'$. Either condition (i) holds ($\alpha\not\le\beta$), in which case by Lemma \ref{alphabeta}, $d_{n-k}(x)=0$ for all $x$ of filtration degree $k$, or else condition (ii) holds. In this case, by Lemma \ref{jackolemma} $d_{n-k}$ is a \jol map, and so $d_{n-k}(x)=0$ for all $x$ of filtration degree $k$, as no top-to-bottom differential is possible. In either case, no differential ever hits the filtration degree of $\beta$. By Lemma \ref{seqs}, this means that $\Hbb(X_n)\to \H\bullet\bullet(X_{n-1})$ is surjective. Now consider the long exact sequence in cohomology corresponding to the cofiber sequence $X_{n-1}\xrightarrow{i}X_n\to \S pq$. 
		\[H^{\bullet-1,\bullet}(X_n)\to H^{\bullet-1,\bullet}(X_{n-1})\xrightarrow{0}\Hbb(\Spq)\to\Hbb(X_n)\to \Hbb(X_{n-1})\xrightarrow{0}\dots\]
		This decomposes into \textit{short} exact sequences
		\[0\to \Sigma^{p,q}\Mt\to\Hbb(X_n)\to\Hbb(X_{n-1})\to0\]
        for all $\bullet$. Since $\Hbb(X_{n-1})$ is free by Theorem \ref{freeness} (or alternatively since $\Mt$ is injective by \cite{clover}) the short exact sequence of modules is split, and the theorem is proved.
	\end{proof}

\section{Grassmannians $\Gr_k(\R^{n,1})$}\label{knone}
We are now ready to tackle the Grassmannian. We begin by introducing a statistic of Young diagrams which will be useful for classifying Schubert cells in the family of spaces $\Gr_k(\R^{n,1})$. Fix $n$ and $k<n$. Given a partition $\lambda=(\lambda_1,\dots,\lambda_k)$ with $\lambda_i\le \lambda_{i+1}$, define trace($\lambda)=\#\{i:\lambda_i\ge k-i+1\}$. Visually, this is the number of squares lying on the diagonal of a Young diagram of this partition. See Figure \ref{fig:traces} for examples. Recall that the \textbf{jump sequence} $\underline{j}=[j_1,\dots,j_k]$ corresponding to a partition $\lambda$ is given by $j_i=\lambda_i+i$. The values of this sequence tell us where the 1s land in the Schubert cell matrix corresponding to $\lambda$. We can also formulate trace as $\text{trace}(\underline{j})=\#\{i:j_i> k\}$.
\begin{figure}[H]
	\begin{tikzpicture}
		\draw(.6,.6) node {$\yng(3,3,3)$};
		\draw(.6,-.5) node {$\trace(3,3,3)=\trace([4,5,6])=3$};
		\color{red}
		\draw[-] (0,0) -- (1.2,1.2);
	\end{tikzpicture}
	\qquad
	\begin{tikzpicture}
		\draw(.6,.6) node {$\yng(1,3,3)$};
		\draw(.6,-.5) node {$\trace(1,3,3)=\trace([2,5,6])=2$};
		\color{red}
		\draw[-] (0,0) -- (.8,.8);
	\end{tikzpicture}\\
	\begin{tikzpicture}
		\draw(.8,.6) node {$\yng(2,2,2)$};
		\draw(.8,-.5) node {$\trace(2,2,2)=\trace([3,4,5])=2$};
		\color{red}
		\draw[-] (.4,0) -- (1.2,.8);
	\end{tikzpicture}
	\qquad
	\begin{tikzpicture}
		\draw(.6,.4) node {$\yng(1,3)$};
		\draw(.6,-.5) node {$\trace(0,1,3)=\trace([1,3,6])=1$};
		\color{red}
		\draw[-] (0,.0) -- (.4,.4);
	\end{tikzpicture}
	\caption{Traces of Young diagrams and jump sequences}
	\label{fig:traces}
\end{figure}


The following lemma uses trace to compute weight.

\begin{lem}
	\label{tracelemma}
	For the Schubert cell structure on $\Gr_k(\R^{n,1})$ corresponding to the decomposition $\R^{n,1}=\R\triv^{k-1}\oplus\R\sgn^1\oplus \R\triv^{n-k}$, the weight of a Schubert cell  $\Omega_\lambda$ is exactly the trace of $\lambda$.
\end{lem}

\begin{proof}
	Fix a Young diagram $\lambda=(\lambda_1,\dots,\lambda_k)$ fitting inside of a $k\times(n-k)$ grid. Recall it has jump sequence $[\lambda_1+1,\dots,\lambda_k+k]$ corresponding to the location of 1s in the family of matrices whose rowspaces make up $\Omega_\lambda$. One of two cases holds. Either
	\begin{itemize}
		\item No element of the jump sequence equals $k$, in which case each row with a jump exceeding $k$ contains a {\color{red}\scalebox{0.8}{\young(-)}} in dimension $k$ (see the top two examples in Figure \ref{fig:traceweight}). So the number of {\color{red}\scalebox{0.8}{\young(-)}} is $\#\{i: j_i>k\}=\#\{i: j_i\ge k+1\}=\text{trace}(\lambda)$.\\ or 
		\item A $1$ \emph{does} lie in column $k$, say in row $r$. Then $$\text{trace}(\lambda)=\#\{i:j_i> k\}=\#\{i:i> r\}=k-r,$$ which also equals the number of {\color{red}\scalebox{0.8}{\young(-)}} appearing to the left of this $1$ in row $r$ (see bottom examples in Figure \ref{fig:traceweight}). This is because when a matrix which is acted upon and rewritten in canonical Schubert cell form (as on page \pageref{equivschub}), the $r$-th row will be multiplied by -1, changing the sign on each of the $k-r$ variables in that row.
	\end{itemize}
	Recall that the topological dimension of a Schubert cell corresponds to the number of boxes in its Young diagram, and the weight to the number of these {\color{red}\scalebox{0.8}{\young(-)}} boxes. And so $w(\lambda)=\text{trace}(\lambda)$.
\end{proof}

\begin{figure}[h]
	\begin{tikzpicture}
	 \draw(0,1) node {$\left[	\begin{matrix}
		  	\young(+) & \young(+) & {\color{red}\young(-)} & 1 & 0 & 0 \\
		  	\young(+) & \young(+) & {\color{red}\young(-)} & 0 & 1 & 0 \\
			\young(+) & \young(+) & {\color{red}\young(-)} & 0 & 0 & 1
	 	\end{matrix}\right]$};
	\draw(0,0) node {$w(3,3,3)=w([4,5,6])=3$};
	\end{tikzpicture}
	\qquad\quad
	\begin{tikzpicture}
	 \draw(0,1) node {$\left[	\begin{matrix}
		  	\young(+) & 1 & 0 & 0 & 0 & 0 \\
		  	\young(+) & 0 & {\color{red}\young(-)}& \young(+)  & 1 & 0 \\
			\young(+) & 0 & {\color{red}\young(-)}& \young(+) & 0 & 1 \\
	 	\end{matrix}\right]$};
	\draw(0,0) node {$w(1,3,3)=w([2,5,6])=2$};
	\end{tikzpicture}\\
	\begin{tikzpicture}
	 \draw(0,1) node {$\left[	\begin{matrix}
		  	{\color{red}\young(-)} & {\color{red}\young(-)} & 1 & 0 & 0 & 0 \\
		  	\young(+) & \young(+) & 0 & 1 & 0 & 0 \\
			\young(+) & \young(+) & 0 & 0 & 1 & 0
	 	\end{matrix}\right]$};
	\draw(0,0) node {$w(2,2,2)=w([3,4,5])=2$};
	\end{tikzpicture}
	\qquad\quad
	\begin{tikzpicture}
	 \draw(0,1) node {$\left[	\begin{matrix}
		  	1 & 0 & 0 & 0 & 0 & 0 \\
		  	0 &{\color{red}\young(-)} & 1 & 0 & 0 & 0 \\
			0 & \young(+) & 0 & \young(+) & \young(+) & 1
	 	\end{matrix}\right]$};
	\draw(0,0) node {$w(0,1,3)=w([1,3,6])=1$};
	\end{tikzpicture}
	\caption{$k=3$, some cells in $I(++-+++)$ for $\Gr_3(\R^{6,1})$.}
	\label{fig:traceweight}
\end{figure}

Let $\part(p,k,m,t)$ denote the number of partitions of $p$ into $k$ non-negative numbers not exceeding a maximum value $m$, such that the Young diagram corresponding to the partition has trace $t$. Lemma \ref{tracelemma} says that in building $\Gr_k(\R^{n,1})$ using $I(\R\triv^{k-1}\oplus\R\sgn\oplus \R\triv^{n-k})$, the number of $(p,q)$-cells is $\part(p,k,n-k,q)$.\\

\subsection{Example}
	In the ingredients table  $I(++-+++++)$ for $\Gr_3(\R^{8,1})$, the number of $(11,2)$-cells is $\part(11,3,5,2)=2$. This counts \scalebox{0.3}{\yng(1,5,5)} and \scalebox{0.3}{\yng(2,4,5)}, but does not count, for example, \scalebox{0.3}{\yng(1,2,4,4)} (too many terms) or \scalebox{0.3}{\yng(5,6)} (a term exceeds 5) or the trace-3 diagrams \scalebox{0.3}{\yng(3,3,5)} or \scalebox{0.3}{\yng(3,4,4)}, which correspond instead to $(11,3)$-cells.

Now that we know the combinatorics of this construction, we show that the corresponding equivariant cellular spectral sequence collapses.

\begin{lem}
	\label{tracelemma2}
	All differentials are zero in the cellular spectral sequence for $\Gr_k(\R^{n,1})$ corresponding to the ordered decomposition $\R^{n,1}= \R\triv^{k-1}\oplus\R\sgn\oplus\R\triv^{n-k}$.
\end{lem}

\begin{proof}
	In order for a nonzero differential to exist, there must be some Young diagrams $\alpha$ and $\beta$ in bidegrees allowing for a map from the generator of $\alpha$ to the lower cone of $\beta$ (by Lemma \ref{jackolemma}), and also with $\alpha\subset\beta$ (by Lemma \ref{alphabeta}). The bidegree requirement demands that the fixed-set dimension of $\alpha$ is greater than that of $\beta$. That is, denoting the topological degree of $\lambda$ by $|\lambda|$, we must have $|\alpha|-w(\alpha)>|\beta|-w(\beta)$. However, as $\alpha\subset\beta$, the diagram $\beta$ could be built from $\alpha$ by successively adding blocks. Each block would increase topological dimension by one, but could increase the trace (and hence by Lemma \ref{tracelemma} the weight $w$) by at \emph{most} one. Thus $\alpha\subset\beta$ implies $|\alpha|-w(\alpha)\le|\beta|-w(\beta)$. These conflicting requirements show that no differentials are possible if $\Gr_k(\R^{n,1})$ is built in this way. 
\end{proof}

\subsection{Example}
	Suppose $\alpha=(8)=\scalebox{0.4}{\yng(8)}$ so that $|\alpha|-w(\alpha)=8-1$. 
	\begin{itemize}
		\item If $\beta=(1,8)=\scalebox{0.4}{\yng(1,8)}$, then although $\alpha\subseteq \beta$, there is no differential to the filtration of $\beta$ as $8-1\not> |\beta|-w(\beta)=9-1$. That is, $\theta\beta$ is too low for a differential from $\alpha$ to reach it.
		\item On the other hand if $\beta=(3,3,3)=\scalebox{0.4}{\yng(3,3,3)}$ then $|\beta|-w(\beta)=9-3<8-1$, however, this doesn't fit: $\alpha\not\subseteq\beta$ and so there is still no differential, by Lemma \ref{alphabeta}.
	\end{itemize}

Theorem \ref{kn1thm} is now immediate. We restate it here:

\begin{thm}\label{knoneagain}
	\[\rank_{\Mt}^{p,q}\H \bullet\bullet(\Gr_k(\R^{n,1}))=\part(p,k,n-k,q).\]
\end{thm}
\begin{proof}
	By Lemma \ref{tracelemma}, we have a cellular spectral sequence for $\Gr_k(\R^{n,1})$ with generators on the $E_1$ page corresponding to Young diagrams, with topological dimension given by number of boxes, and weight given by trace. By Lemma \ref{tracelemma2}, this spectral sequence immediately collapses.
\end{proof}

\subsection{Comment}
\label{upside}
In Section \ref{foreshadow}, we observed that the rows of the rank charts of cohomologies $\H\bullet\bullet(\Gr_{k}(\R^{n,1}))$ are palindromes.
We can deduce this from the fact that $\rank_{\Mt}^{p,q}\H\bullet\bullet(\Gr_k(\R^{n,1}))$ counts Young diagrams of $p$ boxes with trace $q$ fitting inside of a $k$-by-$(n-k)$ box. To have trace $q$, a Young diagram must have a $q$-by-$q$ square as its southwest corner, with any additional boxes lying in a region to the north or to the east of this square. For example, considering $\Gr_4(\R^{9,1})$, trace-2 diagrams take the form
\[k=4\left\{\phantom{\young(a,b,c,d)}\right.\!\!\!\!\!\!\!\!
\underbrace{\young(??,??,{\,}\diagup ???,\diagup {\,}???)}_{n-k=5}\subseteq \yng(5,5,5,5)\,.\]
For a given trace $q$, the topological dimension $p$ of a Young diagram (corresponding to the number of boxes) is $p=q^2+\#\scalebox{0.8}{\young(?)}$ where $$0\le \#\scalebox{0.8}{\young(?)}\le(k-q)q+q(n-k-q).$$ If we take the complementary diagram in these north and east regions, we get a diagram of dimension
\[q^2+(k-q)q+q(n-k-q)-\#\young(?)=nq-p.\]
This process is clearly reversible, and forms a bijection. For example Figure \ref{fig:bijection} demonstrates this bijection in $\H\bullet\bullet(\Gr_4(\R^{9,1}))$ to show why $\rank_{\Mt}^{6,2}=\rank_{\Mt}^{2*9-6,2}=\rank_{\Mt}^{12,2}=5$.
\begin{figure}
\begin{tikzpicture}[scale=.28]
	\draw[opacity=.5, fill=gray] (0,0)--(2,0)--(2,2)--(0,2)--(0,0);
	\draw (0,1)--(2,1);
	\draw (1,0)--(1,2);
	\draw (0,0)--(2,2);
	\draw[dashed] (0,2)--(0,4)--(2,4)--(2,2)--(5,2)--(5,0)--(2,0);
	\draw(-1,1) node[left] {\tiny $\yng(2,4)$};
	\draw(2,0)--(4,0)--(4,1)--(2,1);
	\draw(3,0)--(3,1);
	\draw[->] (6,2)--(10,2);
	\draw (8,2) node[above] {\small complement};
	\draw (8,2) node[below] {\small regions};
\end{tikzpicture}
\begin{tikzpicture}[scale=.28]
	\draw[opacity=.5, fill=gray] (0,0)--(2,0)--(2,2)--(0,2)--(0,0);
	\draw (0,1)--(2,1);
	\draw (1,0)--(1,2);
	\draw (0,0)--(2,2);
	\draw[dashed] (0,2)--(0,4)--(2,4)--(2,2)--(5,2)--(5,0)--(2,0);
	\draw(6,2) node[right] {\tiny $\yng(2,2,3,5)$};
	\draw(2,0)--(5,0)--(5,1)--(3,1)--(3,2)--(2,2)--(2,4)--(0,4)--(0,2);
	\draw(1,2)--(1,4);
	\draw(0,3)--(2,3);
	\draw(2,1)--(3,1)--(3,0);
	\draw(4,0)--(4,1);
\end{tikzpicture}\\
\begin{tikzpicture}[scale=.28]
	\draw[opacity=.5, fill=gray] (0,0)--(2,0)--(2,2)--(0,2)--(0,0);
	\draw (0,1)--(2,1);
	\draw (1,0)--(1,2);
	\draw (0,0)--(2,2);
	\draw[dashed] (0,2)--(0,4)--(2,4)--(2,2)--(5,2)--(5,0)--(2,0);
	\draw(-2,1) node[left] {\tiny $\yng(3,3)$};
	\draw[->] (6,2)--(10,2);
	\draw (8,2) node[above] {\small complement};
	\draw (8,2) node[below] {\small regions};
	\draw (2,0)--(3,0)--(3,1)--(2,1);
	\draw (2,2)--(3,2)--(3,1);
\end{tikzpicture}
\begin{tikzpicture}[scale=.28]
	\draw[opacity=.5, fill=gray] (0,0)--(2,0)--(2,2)--(0,2)--(0,0);
	\draw (0,1)--(2,1);
	\draw (1,0)--(1,2);
	\draw (0,0)--(2,2);
	\draw[dashed] (0,2)--(0,4)--(2,4)--(2,2)--(5,2)--(5,0)--(2,0);
	\draw(7,2) node[right] {\tiny $\yng(2,2,4,4)$};
	\draw(2,0)--(4,0)--(4,2)--(2,2);
	\draw(3,0)--(3,2);
	\draw(2,1)--(4,1);
	\draw(0,2)--(0,4)--(2,4)--(2,2);
	\draw(1,2)--(1,4);
	\draw(0,3)--(2,3);
\end{tikzpicture}\\
\begin{tikzpicture}[scale=.28]
	\draw[opacity=.5, fill=gray] (0,0)--(2,0)--(2,2)--(0,2)--(0,0);
	\draw (0,1)--(2,1);
	\draw (1,0)--(1,2);
	\draw (0,0)--(2,2);
	\draw[dashed] (0,2)--(0,4)--(2,4)--(2,2)--(5,2)--(5,0)--(2,0);
	\draw(-2,1.5) node[left] {\tiny $\yng(1,2,3)$};
	\draw[->] (6,2)--(10,2);
	\draw (8,2) node[above] {\small complement};
	\draw (8,2) node[below] {\small regions};
	\draw (0,2)--(0,3)--(1,3)--(1,2);
	\draw (2,1)--(3,1)--(3,0)--(2,0);
\end{tikzpicture}
\begin{tikzpicture}[scale=.28]
	\draw[opacity=.5, fill=gray] (0,0)--(2,0)--(2,2)--(0,2)--(0,0);
	\draw (0,1)--(2,1);
	\draw (1,0)--(1,2);
	\draw (0,0)--(2,2);
	\draw[dashed] (0,2)--(0,4)--(2,4)--(2,2)--(5,2)--(5,0)--(2,0);
	\draw (0,2)--(0,4)--(1,4)--(1,2);
	\draw (0,3)--(2,3)--(2,2);
	\draw (2,1)--(4,1)--(4,0)--(2,0);
	\draw (3,0)--(3,2)--(2,2);
	\draw (3,2)--(4,2)--(4,1)--(5,1)--(5,0)--(4,0);
	\draw(6,2) node[right] {\tiny $\yng(1,2,4,5)$};
\end{tikzpicture}\\
\begin{tikzpicture}[scale=.28]
	\draw[opacity=.5, fill=gray] (0,0)--(2,0)--(2,2)--(0,2)--(0,0);
	\draw (0,1)--(2,1);
	\draw (1,0)--(1,2);
	\draw (0,0)--(2,2);
	\draw[dashed] (0,2)--(0,4)--(2,4)--(2,2)--(5,2)--(5,0)--(2,0);
	\draw(-3,1.5) node[left] {\tiny $\yng(2,2,2)$};
	\draw[->] (6,2)--(10,2);
	\draw (8,2) node[above] {\small complement};
	\draw (8,2) node[below] {\small regions};
	\draw (0,2)--(0,3)--(2,3)--(2,2);
	\draw (1,2)--(1,3);
\end{tikzpicture}
\begin{tikzpicture}[scale=.28]
	\draw[opacity=.5, fill=gray] (0,0)--(2,0)--(2,2)--(0,2)--(0,0);
	\draw (0,1)--(2,1);
	\draw (1,0)--(1,2);
	\draw (0,0)--(2,2);
	\draw[dashed] (0,2)--(0,4)--(2,4)--(2,2)--(5,2)--(5,0)--(2,0);
	\draw(6,2) node[right] {\tiny $\yng(2,5,5)$};
	\draw (0,2)--(0,3)--(2,3)--(2,2);
	\draw (1,2)--(1,3);
	\draw (2,0)--(5,0);
	\draw (2,1)--(5,1);
	\draw (2,2)--(5,2);
	\draw (2,0)--(2,2);
	\draw (3,0)--(3,2);
	\draw (4,0)--(4,2);
	\draw (5,0)--(5,2);
\end{tikzpicture}\\
\begin{tikzpicture}[scale=.28]
	\draw[opacity=.5, fill=gray] (0,0)--(2,0)--(2,2)--(0,2)--(0,0);
	\draw (0,1)--(2,1);
	\draw (1,0)--(1,2);
	\draw (0,0)--(2,2);
	\draw[dashed] (0,2)--(0,4)--(2,4)--(2,2)--(5,2)--(5,0)--(2,0);
	\draw(-3,2) node[left] {\tiny $\yng(1,1,2,2)$};
	\draw[->] (6,2)--(10,2);
	\draw (8,2) node[above] {\small complement};
	\draw (8,2) node[below] {\small regions};
	\draw (0,2)--(0,4)--(1,4)--(1,2);
	\draw (0,3)--(1,3);
\end{tikzpicture}
\begin{tikzpicture}[scale=.28]
	\draw[opacity=.5, fill=gray] (0,0)--(2,0)--(2,2)--(0,2)--(0,0);
	\draw (0,1)--(2,1);
	\draw (1,0)--(1,2);
	\draw (0,0)--(2,2);
	\draw[dashed] (0,2)--(0,4)--(2,4)--(2,2)--(5,2)--(5,0)--(2,0);
	\draw(6,2) node[right] {\tiny $\yng(1,1,5,5)$};
	\draw (0,2)--(0,4)--(1,4)--(1,2);
	\draw (0,3)--(1,3);
	\draw (2,0)--(5,0);
	\draw (2,1)--(5,1);
	\draw (2,2)--(5,2);
	\draw (2,0)--(2,2);
	\draw (3,0)--(3,2);
	\draw (4,0)--(4,2);
	\draw (5,0)--(5,2);
\end{tikzpicture}\\
	\caption{Bijection between $\rank_{\Mt}^{6,2}$ and $\rank_{\Mt}^{12,2}$ in $\H\bullet\bullet(\Gr_4(\R^{9,1}))$.}
	\label{fig:bijection}
\end{figure}

 More generally, this bijection proves: 
\begin{thm}
	\[\rank_{\Mt}^{p,q}\H\bullet\bullet(\Gr_k(\R^{n,1}))=\rank_{\Mt}^{nq-p,q}\H\bullet\bullet(\Gr_k(\R^{n,1})).\]
\end{thm}

It would be nice if this apparent duality could be given a geometric interpretation. We do not know one.

\section{Grassmannians $\Gr_2(\R^{n,2})$}\label{twontwo}

\subsection{Comment}
In this section we work up to a general formula for the cohomology of $\Gr_2(\R^{n,2})$ somewhat slowly, starting with calculations for small $n$ which rely on observations about multiple Schubert cell constructions. We do this until we reach a value of $n$ after which we can use the same construction each time, with no new differentials appearing.\par
This approach -- comparing multiple constructions to deduce unknown differentials -- can be automated to perform further calculations not appearing in this paper. In fact a Sage program generating a fund of computations by investigating \emph{all} possible constructions first motivated these results. We hope to write more about this soon.

\subsection{When $n=3$}\label{n=3}
It happens that  $\Gr_2(\R^{3,2})$ is actually isomorphic to $\Gr_1(\R^{3,1})$ and so we have technically already done this computation in \ref{we1}. Nonetheless, the decomposition $\R^{3,2}=\R\sgn\oplus \R\triv\oplus \R\sgn=\R^{-+-}$ gives

\begin{center}
	$I(+-+)=$
	\begin{tabular}{||c|c|c}
	&&\\
	\hline
	&\scalebox{.5}{\young(-)}\,\,&\scalebox{.5}{\young(-,+)}\,\,\\
	\hline
	$\emptyset$\,&&\\
	\hline
	\hline
	\end{tabular}
\end{center}
with no possible differentials, since $\scalebox{0.4}{\yng(1)}\mapsto \scalebox{0.4}{\yng(1,1)}$ would give a nonzero map in singular cohomology, and so $H^{\bullet,\bullet}\Gr_2(\R^{3,2})=\Mt\oplus\M11\oplus\M21$. By the forgetful long exact sequence in \ref{rholes}, we can also say that each of these three generators maps to the unique Schubert class in their topological dimension. We represent this information by labeling generators in the rank chart by their image under $\psi$.
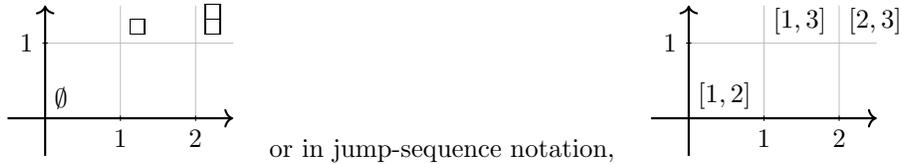
\begin{center}
	\scalebox{.8}{\begin{tikzpicture}[scale=1.0]
		\def\pmax{ 2 }
		\def\qmax{ 1 }
		\draw[step=1cm,lightgray,very thin] (0,0) grid (\pmax+0.5,\qmax+.5);
		\draw[thick,->] (0,-0.5) -- (0,.5+\qmax);
		\draw[thick,->] (-0.5,0) -- (.5+\pmax,0);
		\foreach \x in {1,...,\pmax}
			\draw (\x cm,1pt) -- (\x cm,-1pt) node[anchor=north] {$\x$};
		\foreach \y in {1,...,\qmax}
			\draw (1pt,\y cm) -- (-1pt,\y cm) node[anchor=east] {$\y$};
		\draw (0,0) node[above right] {$\emptyset$};
		\draw (1,1) node[above right] {\scalebox{.5}{\yng(1)}};
		\draw (2,1) node[above right] {\scalebox{.5}{\yng(1,1)}};
	\end{tikzpicture}}
	\quad or in jump-sequence notation, \quad
	\scalebox{.8}{\begin{tikzpicture}[scale=1.0]
		\def\pmax{ 2 }
		\def\qmax{ 1 }
		\draw[step=1cm,lightgray,very thin] (0,0) grid (\pmax+0.5,\qmax+.5);
		\draw[thick,->] (0,-0.5) -- (0,.5+\qmax);
		\draw[thick,->] (-0.5,0) -- (.5+\pmax,0);
		\foreach \x in {1,...,\pmax}
			\draw (\x cm,1pt) -- (\x cm,-1pt) node[anchor=north] {$\x$};
		\foreach \y in {1,...,\qmax}
			\draw (1pt,\y cm) -- (-1pt,\y cm) node[anchor=east] {$\y$};
		\draw (0,0) node[above right] {$[1,2]$};
		\draw (1,1) node[above right] {$[1,3]$};
		\draw (2,1) node[above right] {$[2,3]$};
	\end{tikzpicture}}.
\end{center}
\subsection{When $n=4$}
\label{twofourtwo}
Next consider $\Gr_2(\R^{4,2})$. Two ingredients tables are shown in Figure \ref{two_ingredient_tables}.

\begin{figure}[h]
	\small
	\begin{tabular}{||c|c|c|c|c}
	&&&\scalebox{.5}{\young(-,--)}\,\,&\\
	\hline
	&&&&\scalebox{.5}{\young(+-,-+)}\,\,\\
	\hline
	&\scalebox{.5}{\young(-)}\,\,&\scalebox{.5}{\young(+-)}\,\,\scalebox{.5}{\young(-,+)}\,\,&&\\
	\hline
	$\emptyset$\,&&&&\\
	\hline
	\hline
	\end{tabular}\,\,
	\qquad. 
	\begin{tabular}{||c|c|c|c|c}
	&&&&\\
	\hline
	&&\scalebox{.5}{\young(--)}\,\,\scalebox{.5}{\young(-,-)}\,\,&\scalebox{.5}{\young(-,+-)}\,\,&\scalebox{.5}{\young(-+,+-)}\\
	\hline
	&&&&\\
	\hline
	$\emptyset$\,&\scalebox{.5}{\young(+)}\,\,&&&\\
	\hline
	\hline
	\end{tabular}
	\caption{Ingredients tables $I(-+-+)$ and $I(+--+)$ for $\Gr_2(\R^{4,2})$.}
	\label{two_ingredient_tables}
\end{figure}

If we knew every cofiber sequence differential, we could iteratively attach the cells using just one construction, computing the cohomology of the subspaces using the long exact sequence for each cofiber $X_k\inj X_{k+1}\to S^{\alpha_k}$ until arriving at the answer. Considering the first construction, this is straightforward while building the two-skeleton, as no nonzero differentials were possible. However, when attaching the $e^{3,3}$ labeled $\scalebox{0.4}{\yng(1,2)}$ to this two-skeleton whose cohomology must be $\Mt\oplus \M11\oplus(\M21)^{\oplus 2}$, we have a possible differential and, naively, no way to determine whether it is nonzero.

\begin{figure}[h]
	\begin{tikzpicture}[scale=0.8]
		\def\pmax{ 3 }
		\def\qmax{ 3 }
		\draw[step=1cm,lightgray,very thin] (0,0) grid (\pmax+0.75,\qmax+.5);
		\draw[thick,->] (0,-0.5) -- (0,.5+\qmax);
		\draw[thick,->] (-0.5,0) -- (1+\pmax,0);
		\foreach \x in {1,...,\pmax}
			\draw (\x cm,1pt) -- (\x cm,-1pt) node[anchor=north] {$\x$};
		\foreach \y in {1,...,\qmax}
			\draw (1pt,\y cm) -- (-1pt,\y cm) node[anchor=east] {$\y$};
		\draw (0,0) node[above right] {$\emptyset$};
		\draw (1,1) node[above right] {\scalebox{.4}{\yng(1)}};
		\draw (2,1) node[above right] {\scalebox{.4}{{\yng(2)}}};
		\draw (2,1.3) node[above right] {\scalebox{.4}{{\yng(1,1)}}};
		\draw[->] (2.75,1.5)--node[above]{?}(3.2,1.5);
		\color{blue}
		\draw (3,3) node[above right] {\scalebox{.4}{\yng(1,2)}};
		\draw[->] (3.5,1.5)--(3.5,-0.5);
		\draw[->] (3.5,1.5)--(1.5,-0.5);
		\draw[->] (3.5,3.5)--(4,4);
		\draw[->] (3.5,3.5)--(3.5,4);
		\draw (3.8,1.8) node {$\theta\cdot $\scalebox{0.4}{\yng(1,2)}};
	\end{tikzpicture}
	\caption{One stage of the $-+-+$ construction, corresponding to the cofiber sequence for including the 2-skeleton into the 3-skeleton.}
\end{figure}

If the differential is zero, we next attach an $e^{4,2}$ which has no possible differentials for bidegree reasons, and so our answer would be $\Mt\oplus\M11\oplus(\M21)^{\oplus 2}\oplus\M33\oplus\M42$. This is where the second construction comes in. Notice the $\M33$ in our first hypothetical scenario. In the second construction of Figure \ref{two_ingredient_tables}, no chain of events can end with a generator in this bidegree: The $\scalebox{0.4}{\yng(1,2)}$ would have to shift \emph{up}, which could only happen if a later cell of fixed-set dimension 0 were attached (see \cite{buddies} for more details on Kronholm shifts). As this cannot happen, our mystery differential in the first construction must be non-zero, so that after the resulting shift, we have $$\H\bullet\bullet(\Gr_2(\R^{4,2}))=\Mt\oplus\M11\oplus\M21\oplus\M22\oplus\M32\oplus\M42.$$

We have now answered the question of the module structure of $\H\bullet\bullet(\Gr_2(\R^{4,2}))$. We can also ask about the image of these generators under the forgetful map $\psi$ in terms of Schubert elements. Most of the generators of this free module have an obvious image under $\psi$, as there is a unique generator in most dimensions of the non-equivariant cohomology. But there is some room for ambiguity in dimension 2.
\begin{figure}[H]
	\label{n=4}
	\begin{tikzpicture}[scale=1.0]
		\def\pmax{ 4 }
		\def\qmax{ 2 }
		\draw[step=1cm,lightgray,very thin] (0,0) grid (\pmax+0.75,\qmax+.5);
		\draw[thick,->] (0,-0.5) -- (0,.5+\qmax);
		\draw[thick,->] (-0.5,0) -- (1+\pmax,0);
		\foreach \x in {1,...,\pmax}
			\draw (\x cm,1pt) -- (\x cm,-1pt) node[anchor=north] {$\x$};
		\foreach \y in {1,...,\qmax}
			\draw (1pt,\y cm) -- (-1pt,\y cm) node[anchor=east] {$\y$};
		\draw (0,0) node[above right] {$\emptyset$};
		\draw (1,1) node[above right] {\scalebox{.7}{\yng(1)}};
		\draw (2,1) node[above right] {$x$};
		\draw (2,2) node[above right] {$y$};
		\draw (3,2) node[above right] {\scalebox{.7}{\yng(1,2)}};
		\draw (4,2) node[above right] {\scalebox{.7}{\yng(2,2)}};
	\end{tikzpicture}
	\caption{$\H\bullet\bullet(\Gr_2(\R^{4,2}))$. We choose a generator in each bidegree with a free summand. The image under $\psi$ in dimensions other than 2 is unambiguous. What can be said about the choices $\psi(x)$ and $\psi(y)$? (Note we \emph{choose} a generator because for example we could replace $y$ with $y'=y+\rho x$.)}
\end{figure}

We wish to know the images under $\psi$ of $x$ and $y$. The span of $\psi(x)$ and $\psi(y)$ is that of the non-equivariant Schubert classes $\scalebox{.5}{\yng(2)}$ and $\scalebox{.5}{\yng(1,1)}$. But it isn't clear yet who is sent where. Consider the inclusion $\Gr_2(\R^{+--})\xrightarrow{i} \Gr_2(\R^{+--+})$. Note that from Figure \ref{two_ingredient_tables}, $\sfrac{\Gr_2\R^{4,2}}{\Gr_2\R^{3,2}}$ can be built from a point and cells of weight two in such a way that $\rH21(\sfrac{\Gr_2\R^{4,2}}{\Gr_2\R^{3,2}})=0$. We have long exact sequences in both equivariant and singular cohomology:
\begin{center}
	\begin{tikzpicture}[node distance=2cm, descr/.style={fill=white, inner sep=2.5pt}]
		\def\hscale{4.5}
		\def\vscale{2}
		\def\hbuff{1.2}
		\def\vbuff{.3}
		\draw(0,0) node {$\rH21(\Gr_2\R^{3,2})$};
		\draw(0,\vscale) node {$\rH21(\Gr_2\R^{4,2})$};
		\draw(\hscale,0) node {$H^2_\text{sing}(\Gr_2\R^{3,2})$};
		\draw(\hscale,\vscale) node {$H^2_\text{sing}(\Gr_2\R^{4,2})$};
		\draw[->](\hbuff,\vscale) to node[descr] {$\psi$} (\hscale-\hbuff,\vscale);
		\draw[->](\hbuff,0) to node[descr] {$\psi$} (\hscale-\hbuff,0);
		\draw[->](0,\vscale-\vbuff) to node[descr] {$i^*$} (0,\vbuff);
		\draw[->](\hscale,\vscale-\vbuff) to node[descr] {$i^*$} (\hscale,\vbuff);
		\draw(0,2*\vscale) node {$\rH21(\sfrac{\Gr_2\R^{4,2}}{\Gr_2\R^{3,2}})=0$};
		\draw[->](0,2*\vscale-\vbuff) to (0,\vscale+\vbuff);
		\draw(.5,\vscale-.5) node {$x$};
		\draw(.5,.5) node {\tiny $\yng(1,1)$};
		\draw(\hscale-.5,.5) node {\tiny $\yng(1,1)$};
		\draw(\hscale-.5,\vscale-.5) node {\tiny $\yng(1,1)$};
		\draw(\hscale+.7,\vscale-.5) node {\tiny $\yng(1,1)+\yng(2)$};
	\end{tikzpicture}
\end{center}
In the diagram for $\rH21$, since $i^*$ is injective, the element $x$ is sent to $\scalebox{0.5}{\yng(1,1)}$ (the unique nonzero element -- see Section \ref{n=3}) which is then sent to the corresponding Schubert class by the forgetful map $\psi$. This element $\scalebox{0.5}{\yng(1,1)}$ has two preimages in $H_{\text{sing}}^2\Gr_2(\R^{4,2})$, the elements $\scalebox{0.5}{\yng(1,1)}$ and $\scalebox{0.5}{\yng(1,1)}+\scalebox{0.5}{\yng(2)}$. This leaves four possibilities:
\begin{align}
	\psi(x)={\tiny\yng(1,1)}\text{\qquad and}&\qquad \psi(y)={\tiny \yng(2)}\\
	\psi(x)={\tiny\yng(1,1)}\text{\qquad and}&\qquad \psi(y)={\tiny\yng(1,1)}+{\tiny \yng(2)}\\
	\psi(x)={\tiny\yng(1,1)}+{\tiny\yng(2)}\text{\qquad and}&\qquad \psi(y)={\tiny \yng(2)}\\
	\psi(x)={\tiny\yng(1,1)}+{\tiny\yng(2)}\text{\qquad and}&\qquad \psi(y)={\tiny \yng(1,1)}
\end{align}
Note that actually (1) and (2) are equivalent up to the change of basis $x'=x$ and $y'=y+\tau x$. Cases (3) and (4) are also equivalent under the same change of basis.

To resolve the remaining ambiguity, define the self-map $P:\Gr_2(\R^{4,2})\to \Gr_2(\R^{4,2})$ by $V\mapsto V^\perp$. After $\psi$, $P^*$ maps Young diagrams to their transposes (see Appendix \ref{perp}, Corollary \ref{cor:perpmap}).
This makes scenario $(1)\sim(2)$ impossible by looking at $\H21$: While we would have $P^*(\psi(x))=P^*(\scalebox{0.4}{\yng(1,1)})=\scalebox{0.4}{\yng(2)}$, the element $\scalebox{0.4}{\yng(2)}\not\in\psi(P^*(\H21))=\psi(\H21)$. And so, up to choice of generator in $\H22$, which we will denote by $\scalebox{0.5}{\yng(1,1)}/\scalebox{0.5}{\yng(2)}$, we can represent $\H\bullet\bullet(\Gr_2(\R^{4,2}))$ as
\begin{figure}[H]
	\begin{tikzpicture}[scale=1.5]
		\def\pmax{ 4 }
		\def\qmax{ 2 }
		\draw[step=1cm,lightgray,very thin] (0,0) grid (\pmax+0.5,\qmax+.5);
		\draw[thick,->] (0,-0.5) -- (0,.5+\qmax);
		\draw[thick,->] (-0.5,0) -- (.5+\pmax,0);
		\foreach \x in {1,...,\pmax}
			\draw (\x cm,1pt) -- (\x cm,-1pt) node[anchor=north] {$\x$};
		\foreach \y in {1,...,\qmax}
			\draw (1pt,\y cm) -- (-1pt,\y cm) node[anchor=east] {$\y$};
		\draw (0,0) node[above right] {$\emptyset$};
		\draw (1,1) node[above right] {\scalebox{.7}{\yng(1)}};
		\draw (2,1) node[above right] {$\scalebox{.7}{\yng(1,1)}+\scalebox{.7}{\yng(2)}$};
		\draw (2,2) node[above right] {\scalebox{.7}{\yng(1,1)}\,/\,\scalebox{.7}{\yng(2)}};
		\draw (3,2) node[above right] {\scalebox{.7}{\yng(1,2)}};
		\draw (4,2) node[above right] {\scalebox{.7}{\yng(2,2)}};
	\end{tikzpicture}
	\caption{The rank table for $\H\bullet\bullet(\Gr_2(\R^{4,2}))$, with generators labeled by their images under $\psi$.}
\end{figure}

Note that this version is indeed compatible with $P^*$. In $\H21$ the involution fixes
${\tiny\yng(1,1)}+{\tiny\yng(2)}={\tiny\yng(2)}+{\tiny\yng(1,1)}$ and in $\H22$, $P^*$ interchanges
 $ {\tiny\yng(1,1)/\yng(2)}$ and ${\tiny\yng(2)/\yng(1,1)}$. And so in addition to knowing the module structure of this cohomology, we know the action of the forgetful map in terms of Schubert classes. This is a first step towards determining the equivariant Schubert calculus, which is outside the scope of this paper.\\

\subsection{$\Gr_2(\R^{n,2})$ for $n=5,6$ or $7$}
$\ $\\

Now things start to become more straightforward.\\

Begin with the following ingredients table for $\Gr_2(\R^{5,2})$:
\begin{center}
	$I(+-+-+)=$
	\begin{tabular}{||c|c|c|c|c|c|c|}
	&&&\scalebox{.5}{\young(-,--)}\,\,&&\scalebox{.5}{\young(+-,+--)}\,\,&\scalebox{.5}{\young(-+-,+-+)}\,\,\\
	\hline
	&&&\scalebox{.5}{\young(-+-)}\,\,&\scalebox{.5}{\young(-,++-)}\,\,\scalebox{.5}{\young(+-,-+)}\,\,&&\\
	\hline
	&\scalebox{.5}{\young(-)}\,\,&\scalebox{.5}{\young(+-)}\,\,\scalebox{.5}{\young(-,+)}\,\,&&&&\\
	\hline
	$\emptyset$\,&&&&&&\\
	\hline
	\hline
\end{tabular}\,.
	
\end{center}
First observe that this construction of $\Gr_2(\R^{+-+-+})$ inherits the Kronholm shift of its subspace $\Gr_2(\R^{+-+-})$. The inclusion $i:\R^{+-+-}\inj\R^{+-+-+}$ induces  
$$i^*: \H21(\Gr_2{\R^{5,2}})\to \H21(\Gr_2{\R^{4,2}})=\Zt.$$ If $\Gr_2(\R^{+-+-+})$ \emph{didn't} also have a differential hitting $\theta\,\scalebox{0.4}{\yng(1,2)}$, then we would have $\H21(\Gr_2{\R^{5,2}})=(\Zt)^2$, giving $i^*$ a nonzero kernel, and also $$\psi:\H21(\Gr_2{\R^{5,2}})\inj H^2_{\text{sing}}(\Gr_2(\R^5))=(\Zt)^2$$ 
would be an isomorphism. Since in singular cohomology the inclusion induces an isomorphism $i^*:H^{2}_\text{sing}(\Gr_2(\R^5))\to H^{2}_\text{sing}(\Gr_2(\R^4))$, this would be a failure of naturality. Hence we again have nonzero $d:\langle\scalebox{0.4}{\yng(2)},\scalebox{0.4}{\yng(1,1)}\rangle\to\langle\theta\,\scalebox{0.4}{\yng(1,2)}\rangle$.

Since no other possible nonzero differentials arise in the first construction for bidegree reasons, we have re-derived the cohomology deduced in Example \ref{example}. We are now also justified in labeling these generators with their images under $\psi$, since each topological dimension above the second has generators in only one weight.
So we may represent $\H\bullet\bullet(\Gr_2(\R^{5,2}))$ as in Figure \ref{fig:representation}.
\begin{figure}[h]
	\begin{tikzpicture}[scale=1.5]
		\def\pmax{ 6 }
		\def\qmax{ 3 }
		\draw[step=1cm,lightgray,very thin] (0,0) grid (\pmax+0.5,\qmax+.5);
		\draw[thick,->] (0,-0.5) -- (0,.5+\qmax);
		\draw[thick,->] (-0.5,0) -- (.5+\pmax,0);
		\foreach \x in {1,...,\pmax}
			\draw (\x cm,1pt) -- (\x cm,-1pt) node[anchor=north] {$\x$};
		\foreach \y in {1,...,\qmax}
			\draw (1pt,\y cm) -- (-1pt,\y cm) node[anchor=east] {$\y$};
		\draw (0,0) node[above right] {$\emptyset$};
		\draw (1,1) node[above right] {\scalebox{.5}{\yng(1)}};
		\draw (2,1) node[above right] {$\scalebox{.5}{\yng(1,1)}+\scalebox{.5}{\yng(2)}$};
		\draw (2,2) node[above right] {\scalebox{.5}{\yng(1,1)}/\scalebox{.5}{\yng(2)}};
		\draw (3,2) node[above right] {\scalebox{.5}{\yng(3)},\,\,\scalebox{.5}{\yng(1,2)}};
		\draw (4,2) node[above right] {\scalebox{.5}{\yng(1,3)},\,\,\scalebox{.5}{\yng(2,2)}};
		\draw (5,3) node[above right] {\scalebox{.5}{\yng(2,3)}};
		\draw (6,3) node[above right] {\scalebox{.5}{\yng(3,3)}};
	\end{tikzpicture}
	\caption{A representation of $\H\bullet\bullet(\Gr_2(\R^{5,2}))$}
	\label{fig:representation}
\end{figure}

We will not continue further with these forgetful map calculations, but see Section \ref{warning} for further discussion of difficulties with this question.

As we continue to investigate $\Gr_2(\R^{n,2})$ for larger $n$, we will see that no new differentials ever arise if we use $I(+-+-+\ldots+)$. At first this is trivial. We switch to jump sequence notation for space reasons, omitting the square brackets but parenthesizing a few elements to discuss. For $\Gr_2(\R^{6,2})$, the ingredients table is
\begin{center}
	$I(+-+-++)=$
	\begin{tabular}{||p{9pt}|p{9pt}|p{18pt}|p{15pt}|p{9pt}|p{9pt}|p{9pt}|p{9pt}|p{9pt}|}
		&&&&&&&&5,6\\
		\hline
		&&&$\!\!(2,4)$&&3,5&3,6   4,5&4,6&\\
		\hline
		&&&1,5&1,6   2,5   3,4&2,6&&&\\
		\hline
		&1,3&(1,4) (2,3)&&&&&&\\
		\hline
		1,2&&&&&&&&\\
		\hline
		\hline
		$_0$&&$_2$&&$_4$& &$_6$&&$_8$
	\end{tabular}
\end{center}
and for $\Gr_2(\R^{7,2})$, the ingredients table is
\begin{center}
	$I(+-+-+++)=$
	\scalebox{.9}{\begin{tabular}{||p{10pt}|p{12pt}|p{18pt}|p{18pt}|p{12pt}|p{12pt}|p{12pt}|p{12pt}|p{12pt}|p{12pt}|p{12pt}|}
		&&&&&&&&5,6&5,7&6,7\\
		\hline
		&&&$\!\!(2,4)$&&3,5&3,6   4,5&3,7   4,6&4,7&&\\
		\hline
		&&&1,5&1,6   2,5   3,4&1,7   2,6&2,7&&&&\\
		\hline
		&1,3&(1,4) (2,3)&&&&&&&&\\
		\hline
		1,2&&&&&&&&&&\\
		\hline
		\hline
		$_0$&&$_2$&&$_4$& &$_6$&&$_8$&&$_{10}$
	\end{tabular}}\,.
\end{center}
The only possible differentials, just for bidegree reasons, would occur between the parenthetical entries. That is, with the exception of $[2,4]=\scalebox{0.4}{\yng(1,2)}$, no other generator has a weight high enough that its lower cone falls within range of a possible differential. Again by naturality this one Kronholm shift occurs, and we have our answer.

\subsection{When $n=8$}\label{twoeighttwo}
$\ $\\
However, when we get to $\Gr_2(\R^{8,2})$, we have $I(+-+-++++)=$
\begin{center}
	\scalebox{0.9}{\begin{tabular}{||p{10pt}|p{10pt}|p{10pt}|p{10pt}|p{10pt}|p{10pt}|p{10pt}|p{18pt}|p{18pt}|p{10pt}|p{10pt}|p{10pt}|p{10pt}|}
		&&&&&&&&(5,6)&5,7&5,8   6,7&6,8&7,8\\
		\hline
		&&&2,4&&3,5&3,6   4,5&3,7   4,6&3,8   4,7&4,8&&&\\
		\hline
		&&&1,5&1,6   2,5   3,4&1,7   2,6&1,8   2,7&$\!(2,8)$&&&&&\\
		\hline
		&1,3&1,4   2,3&&&&&&&&&&\\
		\hline
		1,2&&&&&&&&&&&&\\
		\hline
		\hline
		$_0$&&$_2$&&$_4$& &$_6$&&$_8$&&$_{10}$&&$_{12}$
	\end{tabular}}
\end{center}
with a differential possible (at least in terms of the bigrading) from $[2,8]$ to $ \theta[5,6]$, or in Young notation, $\scalebox{.4}{\yng(1,6)}\mapsto\theta\,\, \scalebox{.4}{\yng(4,4)}$. But notice that whereas our nonzero differential back in Section \ref{twofourtwo} had both $\scalebox{.4}{\yng(1,1)}$ and $\scalebox{.4}{\yng(2)}$ fitting inside of $\scalebox{.4}{\yng(1,2)}$, that is not the case here.\par

This is significant because as seen in Section \ref{we2}, containment of subvarieties corresponds to containment of Young diagrams. As $\scalebox{.4}{\yng(1,6)}\not\subset \scalebox{.4}{\yng(4,4)}$ or equivalently in jump sequence notation, as $[2,8]\not\prec[5,6]$, we know that $X_{[2,8]}\not\subseteq X_{[5,6]}$ and so by Theorem \ref{fixedthm}, attaching $\Omega_{[5,6]}$ creates no nonzero differentials.

In fact, this generalizes for $\Gr_2(\R^{n,2})$ with $n\ge 8$. If we chose the identification $\R^{n,2}=\R^{+-+-}\oplus(\R^+)^{n-4}$, the representation cell structure and hence the ingredients table $I(\R^{+-+-}\oplus(\R^+)^{n-4})$ is as follows. A jump sequence $[j_1,j_2]$, will give rise to a cell of topological dimension $(j_1-1)+(j_2-2)$ and by observation will have weight $w([j_1,j_2])=$

\[\begin{cases}
	1 \quad\text{if } [j_1,j_2]=[1,3],[1,4]\text{ or }[2,3]& \text{from}\quad\ds{\small\left[{1\atop 0}{\atop\young(-)\,1}\right]},{\small\left[{1\atop 0}{\atop\young(+-)\,1}\right]},{\small\left[{\young(-)\,1\,\,\,\,\atop\young(+)\,0\,1}\right]}\\
	\\
	2 \quad\text{if } [j_1,j_2]=[3,4]\text{ or }[2,\ge5]& \text{from}\quad \ds{\small\left[{\young(+-)\,1 \atop\young(+-)\,0}{\atop 1}\right]},{\small\left[{\young(-)\,1 \atop\young(+)\,0}{\atop \young(+-+)\,\dots\young(+)\,1}\right]}\\
	\\
	2 \quad\text{if } [j_1,j_2]=[1,\ge 5]& \text{from}\quad \ds{\small \left[{1\atop 0}{\atop\young(-+-+)\,\dots\young(+)\,1}\right]}\\
	\\
	3 \quad\text{if } [j_1,j_2]=[2,4]\text{ or }[3,\ge5]& \text{from}\quad \ds{\small\left[{\young(-)\,1 \atop\young(-)\,0}{\atop \young(-)\,1}\right]},{\small\left[{\young(+-)\,1 \atop\young(+-)\,0}{\atop \young(-+)\,\dots\young(+)\,1}\right]}\\
	\\
	3 \quad\text{if } j_1=4& \text{from}\quad \ds{\small\left[{\young(-+-)\,1 \atop\young(+-+)\,0}{\atop \young(+)\,\dots\young(+)\,1}\right]}\\
	\\
	4 \quad\text{if } j_1\ge 5& \text{from}\quad\ds{\small\left[{\young(+-+-+)\,\dots\young(+)\,1 \atop\young(+-+-+)\,\dots\young(+)\,0}{\atop \young(+)\,\dots\young(+)\,1}\right]}.
\end{cases}
\]
For example $\Gr_2(\R^{10,2})$ has ingredients $I(+-+-++++++)=$

\begin{center}
	\scalebox{.8}{\begin{tabular}{||p{10pt}|p{10pt}|p{18pt}|p{18pt}|p{10pt}|p{10pt}|p{10pt}|p{10pt}|p{16pt}|p{16pt}|p{16pt}|p{16pt}|p{16pt}|p{16pt}|p{16pt}|p{16pt}|p{16pt}}
		&&&&&&&&&&&&&&&&\\
		\hline
		&&&&&&&&5,6&5,7&5,8   6,7&5,9   6,8&5,10   6,9   7,8&6,10   7,9&7,10   8,9&8,10&9,10\\
		\hline
		&&&(2,4)&&3,5&3,6   4,5&3,7   4,6&3,8   4,7&3,9   4,8&3,10   4,9&4,10&&&&&\\
		\hline
		&&&1,5&1,6   2,5   3,4&1,7   2,6&1,8   2,7&1,9   2,8&1,10   2,9&2,10&&&&&&&\\
		\hline
		&1,3&(1,4) (2,3)&&&&&&&&&&&&&&\\
		\hline
		1,2&&&&&&&&&&&&&&&&\\
		\hline
		\hline
		$_0$&&&&&$_5$&&&&&$_{10}$&&&&&$_{15}$
	\end{tabular}}\,.\\
\end{center}
$\ $\\
For generators above topological dimension $3$, the only possible differentials (that is, possible with respect to bidegree) supported by these generators would be maps from $\alpha$ in bidegree $(x,2)$ to $\theta \beta$ for generators $\beta$ in $(x+1,4)$. These $\alpha$ will have jump sequences $[1,x+2]$ or $[2,x+1]$, while the $\beta$'s jump sequence could be $[5,x-1]$, $[6,x-2]$, $[7,x-3]$ etc. In any case, the second number in the jump sequence of each $\alpha$ will be larger than that of a corresponding $\beta$, so there is no dominance relation, or in terms of Young diagrams, $\alpha \not\subseteq \beta$. Now by Theorem \ref{fixedthm} this differential is actually zero.\par 

And so with the exception of the lone nonzero differential to $\theta[2,4]$, i.e. $\theta\,\scalebox{.4}{\yng(1,2)}$, all differentials are zero. We can now count the number of generators ending up in each bidegree. Row by row, if $M=\H\bullet\bullet(\Gr_2(\R^{n,2}))$ for $n\ge 8$,

\begin{figure}[H]
\centering
\begin{minipage}{.4\textwidth}
  \centering
  \begin{align*}
	\rank^{p,0}_{\Mt}M&=\begin{cases}
		1&p=0\\
		0&\text{else}.
	\end{cases}\\
	\rank^{p,1}_{\Mt}M&=\begin{cases}
		1&p=1,2\\
		0&\text{else}
	\end{cases}
  \end{align*}
\end{minipage}%
\begin{minipage}{.6\textwidth}
  \centering
  \begin{align*}
  	\rank^{p,2}_{\Mt}M&=\begin{cases}
  		3 & p=4\\
  		2 & p=3\text{\quad or \quad}5\le p\le n-2\\
  		1 & p=2,n-1\\
  		0 & \text{else}
  	\end{cases}\\
	\rank^{p,3}_{\Mt}M&=\begin{cases}
		2 & 6\le p\le n\\
		1 & p=5,n+1\\
		0 & \text{else}
	\end{cases}\\
	\rank^{p,4}_{\Mt}M&=\begin{cases}
		\lceil\frac {p-7}2\rceil & 8\le p\le n+1\\
		n-1-\lceil\frac p2\rceil & n+2\le p\le 2n-4\\
		0 & \text{else}
	\end{cases}
  \end{align*}
\end{minipage}
\end{figure}

This can also be rewritten to obtain the equally unattractive formula of Theorem \ref{uglyformula}.

\subsection{Warning}
	\label{warning}
	We must be careful not to get carried away in assuming that the images under $\psi$ of these generators correspond to the Schubert cells which are their ``reason'' for appearing where they do in cohomology\footnote{Notice that we stopped labeling generators with their forgetful images at $\Gr_2(\R^{5,2})$.}. For example, in constructing $\Gr_2(\R^{3,1})$, we have $I(-++)=$
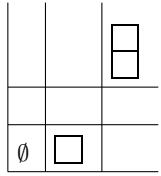
	\begin{center}
		\begin{tikzpicture}
			\draw (0.0,0.0) node[above right] {\small$\emptyset$};
			\draw (0.5,0.0) node[above right] {\small$\yng(1)$};
			\draw (1.25,1.125) node[above right] {\small$\yng(1,1)$};
			\draw[-] (0.0,0)--(0.0,2.25);
			\draw[-] (0.5,0)--(0.5,2.25);
			\draw[-] (1.25,0)--(1.25,2.25);
			\draw[-] (0,0.0)--(2.0,0.0);
			\draw[-] (0,0.625)--(2.0,0.625);
			\draw[-] (0,1.125)--(2.0,1.125);
		\end{tikzpicture}\
	\end{center}
	which must shift to
	
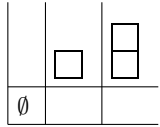
\begin{center}
		\begin{tikzpicture}
			\draw (0.0,0.0) node[above right] {\small$\emptyset$};
			\draw (0.5,0.5) node[above right] {\small$\yng(1)$};
			\draw (1.25,0.5) node[above right] {\small$\yng(1,1)$};
			\draw[-] (0.0,0)--(0.0,1.625);
			\draw[-] (0.5,0)--(0.5,1.625);
			\draw[-] (1.25,0)--(1.25,1.625);
			\draw[-] (0,0.0)--(2.0,0.0);
			\draw[-] (0,0.5)--(2.0,0.5);
		\end{tikzpicture}\
	\end{center}
	because, for example, $\Gr_2(\R^{3,1})\iso \Gr_1(\R^{3,1})$. (See Section \ref{we1}.) However when we proceed to build $\Gr_2(\R^{4,2})$ by attaching the remaining cells of $I(-++-)=$
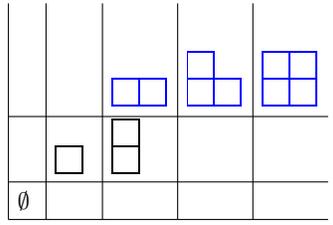
	\begin{center}
		\begin{tikzpicture}[scale=1.2]
			\draw (0.0,0.0) node[above right] {\small$\emptyset$};
			\draw (0.5,0.5) node[above right] {\small$\yng(1)$};
			\draw (1.25,0.5) node[above right] {\small$\yng(1,1)$};
			\draw[-] (0.0,0)--(0.0,2);
			\draw[-] (0.5,0)--(0.5,2);
			\draw[-] (1.25,0)--(1.25,2);
			\draw[-] (2.25,0)--(2.25,2);
			\draw[-] (3.25,0)--(3.25,2);
			\draw[-] (0,0.0)--(4.25,0.0);
			\draw[-] (0,0.5)--(4.25,0.5);
			\draw[-] (0,1.37)--(4.25,1.37);
			\color{blue}
			\draw (1.25,1.4) node[above right] {\small$\yng(2)$};
			\draw (2.25,1.4) node[above right] {\small$\yng(1,2)$};
			\draw (3.25,1.4) node[above right] {\small$\yng(2,2)$};
		\end{tikzpicture}\,,
	\end{center}
	there are no possible nonzero differentials, and so we may be tempted to keep these Young diagram labelings, and assert that the forgetful map $$\psi:\H\bullet\bullet(\Gr_2(\R^{4,2}))\to H_\text{sing}(\Gr_2(\R^{4}))$$ sends these generators to the non-equivariant Schubert cell corresponding to those Young diagrams. In fact we know from Section \ref{twofourtwo} that this is \emph{false}. Each attachment of a new cell raises doubts as to the forgetful image of the cohomology. Put another way, maps in equivariant cohomology induced by inclusion of Grassmannians need not respect Schubert symbols in singular cohomology. 
	
	For this reason, when we looked at the ingredients table for $\Gr_2(\R^{10,2})$ above, while we know all of the differentials, and thus the ranks in each dimension, we don't (yet) have a good reason to assign to these generators the Schubert symbols we would naturally wish to.

\section{Some infinite Grassmannians}\label{infgrass}
We can now also deduce the cohomologies of the analogous infinite Grassmannians which follow from these results. As a consequence of Theorem 6.2 of \cite{buddies}, the infinite Grassmannian $\Gr_2(\R^{\infty,2})$ will have free cohomology and a zero $\lim^1$ term, and so from Theorem \ref{uglyformula}, the rank table begins

\begin{figure}[H]
	\begin{tikzpicture}[scale=0.7]
		\def\pmax{13}
		\def\qmax{4}
		\draw[step=1cm,lightgray,very thin] (0,0) grid (\pmax+0.5,\qmax+.5);
		\draw[thick,->] (0,-0.5) -- (0,.5+\qmax);
		\draw[thick,->] (-0.5,0) -- (.5+\pmax,0);
		\draw (7.18,3.3) node {$2$};
		\draw (2.18,1.3) node {$1$};
		\draw (6.18,2.3) node {$2$};
		\draw (7.18,2.3) node {$2$};
		\draw (6.18,3.3) node {$2$};
		\draw (2.18,2.3) node {$1$};
		\draw (1.18,1.3) node {$1$};
		\draw (3.18,2.3) node {$2$};
		\draw (0.18,0.3) node {$1$};
		\draw (8.18,2.3) node {$2$};
		\draw (4.18,2.3) node {$3$};
		\draw (5.18,3.3) node {$1$};
		\draw (8.18,3.3) node {$2$};
		\draw (5.18,2.3) node {$2$};
		\draw (8.18,4.3) node {$1$};
		\draw (8,0.1)--(8,-0.1) node[below] {$8$};
		\draw (9.18,2.3) node {$2$};
		\draw (10.18,2.3) node {$2$};
		\draw (11.18,2.3) node {$2$};
		\draw (12.18,2.3) node {$2$};
		\draw (13.55,2.3) node {$2\dots$};
		\draw (9.18,3.3) node {$2$};
		\draw (10.18,3.3) node {$2$};
		\draw (11.18,3.3) node {$2$};
		\draw (12.18,3.3) node {$2$};
		\draw (13.55,3.3) node {$2\dots$};
		\draw (9.18,4.3) node {$1$};
		\draw (10.18,4.3) node {$2$};
		\draw (11.18,4.3) node {$2$};
		\draw (12.18,4.3) node {$3$};
		\draw (13.55,4.3) node {$3\dots$};
	\end{tikzpicture}
\end{figure}

and then as dimension increases,

\begin{thm}
	 For $p\ge8$,

	\[\rank_{\Mt}\H p{\bullet}(\Gr_2(\R^{\infty,2}))=\begin{cases}
		\lceil\frac {p-7}2\rceil &\bullet=4\\
		2 & \bullet=3\\
		2 & \bullet=2\\
		0 &\text{else.}
	\end{cases}
	\]
 \end{thm}

Similarly, as a consequence of Theorem \ref{knoneagain},

\begin{thm}
	\[\rank_{\Mt}\H pq(\Gr_k(\R^{\infty,1}))=\part(p,k,\infty,q).\]
\end{thm}

Using the logic from Comment \ref{upside} which considers a $q$-by-$q$ square with a region north and a region to the east, this formula can also be expressed (for $p\ge q^2$)
\[\rank_{\Mt}\H pq(\Gr_k(\R^{\infty,1}))=\sum_{i=1}^{q(k-q)}\part(i,k-q,q,\ast)\part(p-q^2-i,q,\ast,\ast)\]
where the $\ast$ denotes omitting that restriction, so $\part(a,b,c,\ast)$ counts partitions of $a$ into $b$ parts not exceeding $c$ but having any trace, and $\part(a,b,\ast,\ast)$ counts partitions of $a$ into $b$ numbers of any size and trace.

\section{Complex Grassmannians}
\label{complexsection}
Modified statements of the results of this paper also apply to complex Grassmannians. Note that while in the real case, a Schubert cell indexed by a partition $\lambda$ of some integer $|\lambda|$ corresponds to a $|\lambda|$-disc, that is, $\Omega_\lambda(\R)\simeq e^{|\lambda|}$, in the complex case, each complex variable contributes two real dimensions: $\Omega_\lambda(\C)\simeq e^{2|\lambda|}$.\par

Define $\C\triv$ and $\C_{\text{sgn}}$ analogously so in $\C_{\text{sgn}}$ we have $z\mapsto -z$, and then let $\C^{p,q}=\C\triv^{p-q}\oplus \C_{\text{sgn}}^q$ as in the real case. For each partition $\lambda$ fitting inside a $k$-by-$(n-k)$ rectangle, whenever $\Gr_k(\R^{p,q})$ has $\Omega_\lambda(\R)\iso e^{a,b}$, the complex Grassmannian $\Gr_k(\C^{p,q})$ has $\Omega_\lambda(\C)\iso e^{2a,2b}$. Recall that a differential $d:\Sigma^{a,b}\Mt^+\to\Sigma^{a',b'}\Mt^-$ is possible only when $a'-b'<a-b$. Because this is equivalent to the inequality $2a'-2b'<2a-2b$, the \emph{possible} differentials on the $E_1$ page of a cellular filtration spectral sequence of a complex Grassmannian occur between the same Schubert cell filtrations as in the real case. And if the same possible differentials are, in fact nonzero, the Kronholm shifts (see formulas in \cite{buddies}) will be twice as large in the complex case, meaning the same possible differentials present themselves on $E_2$, and so on.

\subsection{Warning}
Because of the essentially un-geometric approach to differentials in this paper, we have no reason to claim that a nonzero differential in the real case must correspond to a nonzero differential in the complex case, or vice versa.\\

However, because the arguments in Lemma \ref{tracelemma} are almost identical with complex variables, we may conclude that in the $\C^{n,1}=\C\triv^{k-1}\oplus \C_{\text{sgn}}\oplus \C^{n-k}\triv$ construction of $\Gr_k(\C^{n,1})$, the trace of a Schubert cell determines its weight: $\Omega_\lambda(\C)\simeq e^{2|\lambda|,2\trace{\lambda}}$. As the complex Grassmannian still satisfies the assumptions of Theorem \ref{fixedthm}, we have the analogous theorem:

\begin{thm}\label{knoneagaincomplex} If $p$ or $q$ is odd, $\rank_{\Mt}^{p,q}\H \bullet\bullet(\Gr_k(\C^{n,1}))=0$, while
	\[\rank_{\Mt}^{2p,2q}\H \bullet\bullet(\Gr_k(\C^{n,1}))=\part(p,k,n-k,q).\]
\end{thm}

Because the arguments of Section \ref{twontwo} are identical if we just double every bidegree, and the ``perp map'' argument in Appendix \ref{perp} applies to both the real and complex case, we can also conclude

\begin{thm}\label{dividebytwo}
	\[\rank^{p,q}_{\Mt}\H\bullet\bullet(\Gr_2(\C^{n,2}))=\begin{cases}
		\rank_{\Mt}^{\frac p2,\frac q2}\H\bullet\bullet(\Gr_2(\R^{n,2}))&\text{$p$ and $q$ even}\\
		0&\text{else}.
	\end{cases}
	\]
\end{thm}

\subsection{Remark}
We can also give $\C$ the conjugation action, $z\mapsto \bar z$. Note that $\C_{\text{conj}}\iso \R^{2,1}$. And so $\Gr_k(\C^n_\text{conj})$ has Schubert cells $\Omega_\lambda\iso e^{2|\lambda|,|\lambda|}$. Purely for degree reasons, no possible differentials $\alpha\to \frac\theta{\rho^i\tau^j}\beta$ exist when $\alpha$ has bidegree $(2x,x)$ and $\beta$ has bidegree $(2y,y)$. Thus the spectral sequence for $\Gr_k(\C^n_{\text{conj}})$ collapses on the first page. Denoting $r_i=\dim H^{2i}_{\text{sing}}(\Gr_k(\C^n);\Zt)$, we have $$\H\bullet\bullet(\Gr_k(\C^n_{\text{conj}}))=\bigoplus_{i=0}^{k(n-k)}(\Sigma^{2i,i}\Mt)^{r_i}.$$

\subsection{Example}
	Consider $\Gr_2(\C^4_\text{conj})$. The action on the Schubert cell 
	\begin{center}
		\begin{tabular}{rl}
				$\Omega_{\scalebox{0.4}{\yng(1,2)}}$
				&=$\left\{\text{rowspace}\left[\begin{matrix}
					z_1&1&0&0\\
					z_2&0&z_3&1\\
				\end{matrix}\right]: z_i\in \C\right\}$\\
				&\\
				&$=\left\{\text{rowspace}\left[\begin{matrix}
					x_1+y_1i&1&0&0\\
					x_2+y_2i&0&x_3+y_3i&1\\
				\end{matrix}\right] :x_i,y_i \in \R\right\}$\\
				&\\
				&$\simeq e^6$
			\end{tabular}
	\end{center}
	sends
	\[\left[\begin{matrix}
					x_1+y_1i&1&0&0\\
					x_2+y_2i&0&x_3+y_3i&1\\
		\end{matrix}\right]
			\mapsto
		\left[\begin{matrix}
						x_1-y_1i&1&0&0\\
						x_2-y_2i&0&x_3-y_3i&1\\
		\end{matrix}\right]
	\]
	and so $\Omega_{\scalebox{0.4}{\yng(1,2)}}\simeq e^{6,3}$. Analogous consideration for the other Schubert cells give a spectral sequence whose $E_1$ page has generators as shown.

	\begin{center}
		\begin{tikzpicture}[scale=0.8]
				\def\pmax{ 8 }
				\def\qmax{ 4 }
				\draw[step=1cm,lightgray,very thin] (0,0) grid (\pmax+0.75,\qmax+.5);
				\draw[thick,->] (0,-0.5) -- (0,.5+\qmax);
				\draw[thick,->] (-0.5,0) -- (1+\pmax,0);
				\foreach \x in {1,...,\pmax}
					\draw (\x cm,1pt) -- (\x cm,-1pt) node[anchor=north] {$\x$};
				\foreach \y in {1,...,\qmax}
					\draw (1pt,\y cm) -- (-1pt,\y cm) node[anchor=east] {$\y$};
				\draw (0,0) node[above right] {$\emptyset$};
				\draw (2,1) node[above right] {\scalebox{.4}{\yng(1)}};
				\draw (4,2) node[above right] {\scalebox{.4}{{\yng(2)}}};
				\draw (4,2.3) node[above right] {\scalebox{.4}{{\yng(1,1)}}};
				\draw (6,3) node[above right] {\scalebox{.4}{\yng(1,2)}};
				\draw (8,4) node[above right] {\scalebox{.4}{\yng(2,2)}};
			\end{tikzpicture}
	\end{center}
As this collapses,
\[\H\bullet\bullet(\Gr_2(\C^4_\text{conj}))=\Mt \oplus \M21 \oplus (\M42)^{\oplus 2}\oplus \M63 \oplus \M84. \]

\subsection{Remark}
Finally, the observations in Section \ref{infgrass} can be similarly duplicated to infinite complex Grassmannians by replacing every $\Sigma^{a,b}\Mt$ with $\Sigma^{2a,2b}\Mt$. This includes the family $\Gr_k(\C_\text{conj}^\infty)$, which satisfies the finite-type condition of \cite{buddies}.

\appendix
\section{The Perp Map}\label{perp}

The goal of this appendix is to show how the map induced in singular cohomology by the  ``perpendicular complement'' map of Grassmannians acts on Schubert symbols. This result is needed in section \ref{twofourtwo}.

Throughout the appendix, let $\F$ denote either $\R$ or $\C$ as desired.

\begin{defn} Given a partition $\lambda=(\lambda_1,\lambda_2,\dots,\lambda_k)$, corresponding to a Schubert cell in $\Gr_k\F^n$, define the \textbf{transpose} $\lambda^T$ \label{transdef} by
\[\lambda^T=\left(\#\{j:\lambda_j> n-k-1\},\#\{j:\lambda_j> n-k-2\},\dots,\#\{j:\lambda_j>1\},\#\{j:\lambda_j>0\}\right).\]
Or more briefly, $$\lambda^T_i=\#\{j:\lambda_j> n-k-i\}\qquad\text{for } 1\le i\le n-k.$$ This partition corresponds to a Schubert cell in $\Gr_{n-k}\F^n$.
\end{defn}
As we might hope, the Young diagram of $\lambda^T$ looks like that of $\lambda$ but reflected across a diagonal.

\subsection{Example}
	If we consider $\scalebox{0.4}{\yng(1,3)}$ as indexing a cell of $\Gr_3(\F^7)$ then

	\begin{align*}
		\scalebox{.4}{\yng(1,3)}^T&=(0,1,3)^T\\
		&=(\#\{i:\lambda_i>7-3-1\},\#\{i:\lambda_i>2\},\#\{i:\lambda_i>1\},\#\{i:\lambda_i>0\})\\
		&=(\#\emptyset,\#\{2\},\#\{2\},\#\{1,2\})\\
		&=(0,1,1,2)\\
		&=\scalebox{.4}{\yng(1,1,2)}.
	\end{align*}

\subsection{Note}
	The four-tuple $(0,1,1,2)$ indexes a cell in $\Gr_{4}(\F^7)$. Had we instead considered $\scalebox{0.4}{\yng(1,3)}$ as indexing a cell of, for example, $\Gr_5(\F^{12})$ then we would have $(0,0,0,1,3)^\perp=(0,0,0,0,1,1,2)$, which indexes a cell of $\Gr_7(\F^{12})$.

Up to this point, when we expressed a Schubert cell as the collection of $k$-planes which are rowspaces of matrices of a certain form, we haven't bothered to make explicit reference to a choice of basis for $\F^n$. Now we will need to. 
\begin{defn}
	If we let $\{e_i\}_{i=1}^n$ be the standard orthonormal basis of $\F^n$, we will wish to denote the \textbf{reverse basis} $\{e_{n-i+1}\}_{i=1}^n$ by $\{\widehat{e}_i\}_{i=1}^n$. We will give subscripts to the $\rs$ operator, so that for a basis $\{b_i\}_{i=1}^n$ and a matrix $M$, $\rs_{\{b_i\}}(M)$ denotes the $\F$-span of the $\{b_i\}$-linear combinations taken from the rows of $M$, that is
	\begin{align*}
		\rs_{\{b_i\}}
			\left[\begin{matrix}
				x_{1,1}&\dots&x_{1,n}\\
				\vdots&\ddots&\vdots\\
				x_{k,1}&\dots &x_{k,n}
			\end{matrix}\right]
			=\left\langle\left\{ \sum_{j=1}^n x_{i,j}b_j \right\}_{i=1}^k\right\rangle_{\F.}
	\end{align*}
\end{defn}

In particular, we have
	\begin{align*}
			\rs_{\{e_i\}}
			\left[\begin{matrix}
				x_{1,1}&\dots&x_{1,n}\\
				\vdots&\ddots&\vdots\\
				x_{k,1}&\dots &x_{k,n}
			\end{matrix}\right]
			&=\rs_{\{\widehat{e}_i\}}
			\left[\begin{matrix}
				x_{1,n}&\dots&x_{1,1}\\
				\vdots&\ddots&\vdots\\
				x_{k,n}&\dots &x_{k,1}
			\end{matrix}\right]\\
			&=\rs_{\{\widehat{e}_i\}}
			\left[\begin{matrix}
				x_{k,n}&\dots&x_{k,1}\\
				\vdots&\ddots&\vdots\\
				x_{1,n}&\dots &x_{1,1}
			\end{matrix}\right].
	\end{align*}
\begin{defn}
Let $\Omega_\lambda$ denote the Schubert cell in the standard basis, while $\widehat\Omega_\lambda$ means the Schubert cell defined with respect to the reverse basis. So if $M$ is a matrix such that $\rs_{\{e_i\}}M\in\Omega_\lambda$, then $\rs_{\{\widehat e_i\}}M\in\widehat\Omega_\lambda$.
\end{defn}

Consider the equivariant homeomorphism
\begin{align*}
	P:\Gr_k(\F^{p,q})&\to \Gr_{p-k}(\F^{p,q})\\
	V&\mapsto V^\perp
\end{align*}
sending each $k$-plane $V\in\Gr_k(\F^{p,q})$ to its perpendicular complement (with respect to the dot product), the ($p-k$)-plane $P(V)=V^\perp\in \Gr_{p-k}(\F^{p,q})$.
It is a useful fact that this is a cellular map under the Schubert construction, mapping Schubert cells in a given flag indexed by $\lambda$ bijectively onto Schubert cells in the reverse flag with transpose Young diagrams: $$P(\Omega_\lambda)=\widehat\Omega_{\lambda^T}.$$
Before proving this fact, we give an example.

\subsection{Example}
	\label{ex:toyperp}
	Take $\Omega_{\,\scalebox{.3}{\yng(1,3)}}$ in $\Gr_2\F^5$. This is the collection of $\F$-planes of the form
	\[V=\rs_{\{e_i\}}
	\left[\begin{matrix}
		x_{1,1}&1&0&0&0\\
		x_{2,1}&0&x_{2,3}&x_{2,4}&1
	\end{matrix}\right]
	\]
	for all $x_{1,1},\,x_{2,1},\,x_{2,3},\,x_{2,4}\in\F$.
	A vector $\vec{y}=(y_1,\dots,y_5)$ is perpendicular to this plane if $y_1x_{1,1}+y_2=0$ and also $y_1x_{2,1}+y_3x_{2,3}+y_4x_{2,4}+y_5=0$. These relations let us write $\vec{y}$ just in terms of $y_1,y_3$ and $y_4$:
	\begin{align*}
		\vec{y}=(y_1,-y_1x_{1,1},y_3,y_4,-(y_1x_{2,1}+y_3x_{2,3}+y_4x_{2,4}))
		&=y_1(1,-x_{1,1},0,0,-x_{2,1})\\
		&+y_3(0,0,1,0,-x_{2,3})\\
		&+y_4(0,0,0,1,-x_{2,4}).
	\end{align*}
	in other words,
	\[\{\vec{y}:\vec{y}\perp V\}=\rs_{\{e_i\}}
	\left[\begin{matrix}
		1&-x_{1,1}&0&0&-x_{2,1}\\
		0&0&1&0&-x_{2,3}\\
		0&0&0&1&-x_{2,4}
	\end{matrix}\right]\in P\left(\Omega_{\,\scalebox{.3}{\yng(1,3)}}\right).
	\]
	Changing to the reverse flag, we have
	\[V^\perp=\rs_{\{\widehat e_i\}}
	\left[\begin{matrix}
		-x_{2,4}&1&0&0&0\\
		-x_{2,3}&0&1&0&0\\
		-x_{2,1}&0&0&-x_{1,1}&1
	\end{matrix}\right]
	\in\widehat\Omega_{\,\scalebox{.3}{\yng(1,1,2)}}=\widehat\Omega_{\left(\scalebox{.3}{\yng(1,3)}\right)^T}\subset \Gr_3(\F^5).
	\]

Since $(V^\perp)^\perp=V$, this map is invertible, and so $\Omega_{\,\scalebox{0.3}{\yng(1,3)}}$ and $\hat\Omega_{\,\scalebox{0.3}{\yng(1,1,2)}}$ are in bijective correspondence.\\

To prove that this phenomenon persists generally, we need two lemmas. The first is fairly obvious.

\begin{defn}
	A \emph{Boolean matrix} is one whose entries consist of 0 and 1.
\end{defn}

\begin{lem}\label{booleanlemma}
	 Boolean matrices with monotone columns are determined by their column sums. Similarly, boolean matrices with monotone rows are determined by their row sums.
\end{lem}
\begin{proof}
	If a boolean matrix $A=[\vec{a}_1\,\,\,\vec{a}_2\,\,\,\dots\,\,\,\vec{a}_k]$ has columns $\vec{a}_i$, each of which is weakly increasing, and the sum of the entries in $\vec{a}_i$ is $s$, then that column's last $s$ entries are 1, and the rest 0. The proofs if we replace ``increasing'' with ``decreasing'' or ``columns'' with ``rows'' are analogous.
\end{proof}

The second lemma is a combinatorial identity.

\begin{lem}
	\label{bijectionlemma}
	Fix $n$ and a partition $\lambda$ having $k<n$ terms not exceeding $n-k$. Let $H$ be the complement of the jump sequence for $\lambda$:
	$$H=[1,n]\setminus\{\lambda_l+l:1\le l\le k\}.$$ 
	Index $H$ by $H=\{h_i\}_{i=1}^{n-k}$ such that $h_i\le h_{i+1}$. Then for each $i$ we have
	\[\#\{j:\lambda_j+j>h_i\text{ and }1\le j\le k\}=\#\{j:\lambda_j>i-1\text{ and } 1\le j\le k\}.\]
\end{lem}	
\begin{proof}
	To prove that these sets have the same size, we will view their cardinalities as the sums of columns in boolean arrays. Define, for $i\in [1,n-k]$ and $j\in[1,k]$ the quantities
	\[\chi_{i,j}=\begin{cases}
		1&\lambda_j+j>h_i\\
		0&\text{else}
	\end{cases}
	\quad\text{and}\quad
	\xi_{i,j}=\begin{cases}
		1&\lambda_j>i-1\\
		0&\text{else}.
	\end{cases}
	\]
	We can now rephrase the lemma as the claim that for all $i$, $\sum_j \chi_{i,j}=\sum_j \xi_{i,j}$. Note that the $(n-k)$-by-$k$ boolean arrays $[\chi]$ and $[\xi]$ are monotonic in both $i$ and $j$, and hence by Lemma \ref{booleanlemma}, determined by these sums of their columns. They are also determined by the sums of their rows. And fixing $j$,
	\begin{align*}
		\sum_i\chi_{i,j}=\#\{i:\lambda_j+j>h_i\}&=\#\{h_i:h_i<\lambda_j+j\}\\
		&=\#\left([1,\lambda_j+j]\setminus{}\{\lambda_l+l\,:\,1\le l\le j\}\right)\\
		&=(\lambda_j+j)-j\\
		&=\lambda_j
	\end{align*}
	while
	\begin{align*}
		\sum_i \xi_{i,j} =\#\{i:\lambda_j>i-1\}=\#\{1,2,\dots,\lambda_j\}=\lambda_j.
	\end{align*}	
	Hence, in fact $\chi_{i,j}\equiv \xi_{i,j}$ for all $i$ and $j$. This proves the lemma.
\end{proof}

\subsection{Example}
	In Example \ref{ex:toyperp}, $n=5$, $k=2$ and $\lambda=[1,3]$ so $H=[1,5]\setminus\{2,5\}=\{1,3,4\}=\{h_1,h_2,h_3\}$, and $\chi_{i,j}\equiv \xi_{i,j}$, each have table $$\begin{tabular}{c|ccc}
	$j\setminus i$ &1&2&3\\
	\hline
	1&$1$ &$0$ &$0$ \\
	2&$1$ &$1$ &$1$ \\
\end{tabular}.$$

For example, $\chi_{3,2}=1$ because $3+2>4$ and $\xi_{3,2}=1$ because $3>3-1$.

And so 
\begin{align*}
	\#\{j:\lambda_j+j>h_1\}=&2=\#\{j:\lambda_j>1-1\}\\
	\#\{j:\lambda_j+j>h_2\}=&1=\#\{j:\lambda_j>2-1\}\\
	\#\{j:\lambda_j+j>h_3\}=&1=\#\{j:\lambda_j>3-1\}.
\end{align*}

\begin{thm}
	\label{perpthm}
	The perp map $P:\Gr_k(\F^n)\to \Gr_{n-k}(\F^n)$ is a homeomorphism sending each Schubert cell $\Omega_\lambda\subseteq \Gr_k(\F^n)$ homeomorphically to $\widehat\Omega_{\lambda^T}\subseteq \Gr_{n-k}(\F^n)$, for each $\lambda$ indexing Schubert cells of $\Gr_k(\F^n)$.
\end{thm}

\begin{proof}
	Let $\lambda=[\lambda_1,\lambda_2,\dots,\lambda_k]$, and $H=[1,n]\setminus\{\lambda_l+l:1\le l\le k\}$. If $V\in \Omega_\lambda$,
	\begin{align*}
		V
		&=\rs_{\{e_i\}}
			\left[\begin{array}{*{20}c}
				x_{1,1}&\dots&x_{1,\lambda_1}&1&0&\dots&0&0&\dots&0\\
				x_{2,1}&\dots&x_{2,\lambda_1}&0&x_{2,\lambda_1+2}&\dots& 1&0&\dots & 0\\
				\vdots&&\vdots&\vdots&\vdots&&\vdots&\vdots&&\vdots\\
				x_{k,1}&\dots&x_{k,\lambda_1}&0&x_{k,\lambda_1+2}&\dots&0&x_{k,\lambda_2+3}&\dots\\
			\end{array}\right]\\
			&=\left\langle\left\{\sum_{{i\in H}\atop{i<\lambda_j+j}}x_{j,i}e_i+e_{\lambda_j+j} \right\}_{j=1}^k\right\rangle_{.}\\		
	\end{align*}
	for some $x_{i,j}\in\F$. A vector $\vec{y}=(y_1,\dots,y_n)$ lies in $P(V)=V^\perp$ if it is perpendicular to each row of the matrix, which gives the relations
	\begin{align*}
		y_{\lambda_1+1}&=-\sum_{i<\lambda_1+1}y_ix_{1,i}\\
		y_{\lambda_2+2}&=-\sum_{{i\in H}\atop{i<\lambda_2+2}}y_ix_{2,i}\\
		&\vdots\\
		y_{\lambda_j+j}&=-\sum_{{i\in H}\atop{i<\lambda_j+j}}y_ix_{j,i}\\
	\end{align*}
	expressible in the standard basis $\ds\{e_i\}_{i=1}^n$,
	\[
	\vec{y}
	=\sum_{i\in H}y_ie_i
	-\sum_{j=1}^k\left(
	\sum_{{i\in H}\atop{i<\lambda_j+j}}y_ix_{j,i}\right)
	e_{(\lambda_j+j)}
	=\sum_{i\in H} y_i\left(e_i-\sum_{\{j:\lambda_j+j>i\}}x_{j,i}e_{\lambda_j+j}\right)
	\]
	And so 
	the vector $\vec{y}$ lies in the span 

\begin{align*}
	\left\langle e_i-\sum_{\{j:\lambda_j+j>i\}}x_{j,i}e_{\lambda_j+j} \right\rangle_{i\in H}
	&=\left\langle \sum_{\{j:\lambda_j+j>i\}}(-x_{j,i}\widehat e_{n-\lambda_j-j+1})+\widehat e_{n-i+1}\right\rangle_{i\in H.}
\end{align*}

	
	This is $\rs_{\{\widehat e_i\}} M$ for an ($n-k$)-by-$n$ matrix $M$.	For each of the $n-k$ values in $H$, we want to count the number of free variables $x_{j,i}$ in that row. Enumerating $H=\{h_i\}_{i=1}^{n-k}$ and fixing $i$, by Lemma \ref{bijectionlemma} and Definition \ref{transdef}, this number is
	\begin{align*}
		\#\{j:\lambda_j+j>h_i\}&=\#\{j:\lambda_j\ge i-1\}\\
		&=\#\{j:\lambda_j\ge n-k-(n-k+1-i)\}\\
		&=(\lambda^T)_{(n-k+1)-i}.
	\end{align*}
	In other words, if $r$ counts up from the bottom row, the $r$th row has $\lambda^T_r$ free variables. 
	
	And so $P(\Omega_\lambda)=\widehat\Omega_{\lambda^T}$.
\end{proof}

Finally, we want to deduce the action of the map induced by $P$ on singular cohomology.

\begin{cor}
	\label{cor:perpmap}
	For $P:\Gr_k(\F^n)\to \Gr_{n-k}(\F^n)$ the perp map, $P^*([\Omega_\lambda])$ is $[\Omega_{\lambda^T}]$ for every $\lambda$ indexing a cell of $\Gr_k(\F^n)$.
\end{cor}

When the perp map is a self-map (and involution) of $\Gr_{k}(\F^{2k,q})$, Corollary \ref{cor:perpmap} can be used to rule out certain possibilities of the forgetful image of generators. We do this in Section \ref{n=4} when computing the images of generators of $\H\bullet\bullet(\Gr_2\R^{4,2})$ under $\psi$.

\begin{proof}[Proof of Corollary \ref{cor:perpmap}]
	In this paper, we refer to elements of cellular cohomology by their dual cells. That is, if $\F=\R$, the element $[\scalebox{0.4}{\yng(1,2)}]\in H^3(\Gr_k\R^n)$ is the class of the cocycle defined by $\Omega_\lambda\mapsto \delta_{\lambda,\scalebox{0.3}{\yng(1,2)}}$. For any two orthonormal bases $\{b_i\}$ and $\{\beta_i\}$ of $\R^n$, there exists a path $\gamma:I\to SO(n)$ so that $\gamma(0)b_i=b_i$ for $1\le i\le n$, $\gamma(1)b_i=\beta_i$ for $1\le i<n$ and either $\gamma(1)b_n=\beta_n$ or $\gamma(1)b_n=-\beta_n$. When $M$ is the matrix corresponding to some partition $\lambda$, the map $\rs_{\{\beta_1,\dots,\beta_n\}}M\mapsto \rs_{\{\beta_1,\dots,-\beta_n\}}M$ is a homeomorphism of $\Omega_\lambda$, and so in either case this one-parameter family shows that $[\widehat \Omega_\lambda]=\pm [\Omega_\lambda]=[\Omega_\lambda]$, as we are working mod 2. Of course, if $\F=\C$, as $U(n)$ is path-connected, we needn't even worry about this $\pm \beta_n$ issue.\par
	And so by Theorem \ref{perpthm}, $P^*([\Omega_\lambda])=[\widehat\Omega_{\lambda^T}]=[\Omega_{\lambda^T}]$.
\end{proof}

\end{document}